\documentclass[letterpaper,11pt]{article}
\usepackage[margin=1in]{geometry}
\usepackage[bookmarks, colorlinks=true, plainpages = false, citecolor = blue,linkcolor=red,urlcolor = blue, filecolor = blue]{hyperref}
\usepackage{url}
\usepackage{amsmath,amsfonts,amsthm,amssymb,bm, verbatim,dsfont,mathtools}
\usepackage{algorithm,algorithmic,color,graphicx,appendix}
\usepackage{subfigure}
\usepackage{etoolbox}
\usepackage{tikz}
\usepackage{xr,xspace}
\usepackage{todonotes}
\usepackage{paralist}
\usepackage{caption}
\theoremstyle{plain}
\newtheorem{theorem}{Theorem}
\newtheorem{lemma}{Lemma}
\newtheorem{proposition}{Proposition}

\theoremstyle{definition}
\newtheorem{definition}{Definition}
\newtheorem{hypothesis}{Hypothesis}

\newtheorem{remark}{Remark}
\newtheorem*{remark*}{Remark}

\renewenvironment{compactitem}
{\begin{itemize}}
{\end{itemize}}

\newcommand{\1}[1]{{\mathbf{1}_{\left\{{#1}\right\}}}}

\newcommand{\Perp}{\indep}
\newcommand{\Hyper}{\text{Hypergeometric}}
\usepackage{xspace,prettyref}
\newcommand{\diverge}{\to\infty}

\newcommand{\iiddistr}{{\stackrel{\text{\iid}}{\sim}}}

\newcommand{\reals}{{\mathbb{R}}}

\newcommand{\naturals}{{\mathbb{N}}}

\newcommand{\eexp}{{\rm e}}

\newcommand{\diff}{{\rm d}}

\newcommand{\Expect}{\mathbb{E}}
\newcommand{\expect}[1]{\mathbb{E}\left[ #1 \right]}
\newcommand{\expects}[2]{\mathbb{E}_{#2}\left[ #1 \right]}

\newcommand{\Prob}{\mathbb{P}}
\newcommand{\prob}[1]{{ \mathbb{P}\left\{ #1 \right\} }}

\newcommand\indep{\protect\mathpalette{\protect\independenT}{\perp}}
\def\independenT#1#2{\mathrel{\rlap{$#1#2$}\mkern2mu{#1#2}}}
\newcommand{\Bern}{{\rm Bern}}
\newcommand{\Binom}{{\rm Binom}}

\newcommand{\eg}{e.g.\xspace}
\newcommand{\ie}{i.e.\xspace}
\newcommand{\iid}{i.i.d.\xspace}
\newrefformat{eq}{(\ref{#1})}
\newrefformat{chap}{Chapter~\ref{#1}}
\newrefformat{sec}{Section~\ref{#1}}
\newrefformat{algo}{Algorithm~\ref{#1}}
\newrefformat{fig}{Fig.~\ref{#1}}
\newrefformat{tab}{Table~\ref{#1}}
\newrefformat{rmk}{Remark~\ref{#1}}
\newrefformat{clm}{Claim~\ref{#1}}
\newrefformat{def}{Definition~\ref{#1}}
\newrefformat{cor}{Corollary~\ref{#1}}
\newrefformat{lmm}{Lemma~\ref{#1}}
\newrefformat{prop}{Proposition~\ref{#1}}
\newrefformat{app}{Appendix~\ref{#1}}
\newrefformat{hyp}{Hypothesis~\ref{#1}}
\newrefformat{thm}{Theorem~\ref{#1}}
\newcommand{\ntok}[2]{{#1,\ldots,#2}}
\newcommand{\pth}[1]{\left( #1 \right)}

\newcommand{\sth}[1]{\left\{ #1 \right\}}
\newcommand{\bpth}[1]{\Bigg( #1 \Bigg)}
\newcommand{\bqth}[1]{\Bigg[ #1 \Bigg]}

\newcommand{\indc}[1]{{\mathbf{1}_{\left\{{#1}\right\}}}}
\newcommand{\Indc}{\mathbf{1}}

\newcommand{\dTV}{d_{\rm TV}}

\newcommand{\tx}{{\widetilde{x}}}

\newcommand{\tA}{{\widetilde{A}}}

\newcommand{\tG}{{\widetilde{G}}}

\newcommand{\tP}{{\widetilde{P}}}

\newcommand{\tS}{{\widetilde{S}}}

\newcommand{\tV}{{\widetilde{V}}}

\newcommand{\tX}{{\widetilde{X}}}

\newcommand{\calA}{{\mathcal{A}}}

\newcommand{\calE}{{\mathcal{E}}}

\newcommand{\calG}{{\mathcal{G}}}

\newcommand{\calX}{{\mathcal{X}}}
\newcommand{\calY}{{\mathcal{Y}}}

\newcommand{\comp}[1]{{#1^{\rm c}}}

\newcommand{\Th}{{\rm th}}

\newcommand{\PDS}{{\sf PDS}\xspace}
\newcommand{\PC}{{\sf PC}\xspace}
\newcommand{\BPDS}{{\sf BPDS}\xspace}
\newcommand{\BPC}{{\sf BPC}\xspace}
\newcommand{\DKS}{{\sf DKS}\xspace}
\newcommand{\PDSR}{{\sf PDSR}\xspace}

\begin{document}

\title{Computational Lower Bounds for Community Detection on Random Graphs}

\date{\today}

\author{Bruce Hajek \and Yihong Wu \and Jiaming Xu\thanks{The authors are with
the Department of ECE, University of Illinois at Urbana-Champaign, Urbana, IL, \texttt{\{b-hajek,yihongwu,jxu18\}@illinois.edu}.}}

\maketitle

\begin{abstract}
This paper studies the problem of detecting the presence of a small dense community planted in a large Erd\H{o}s-R\'enyi random graph $\calG(N,q)$, where the edge probability within the community exceeds $q$ by a constant factor.
Assuming the hardness of the planted clique detection problem, we show that the  computational  complexity of detecting the community exhibits the following phase transition phenomenon: As the graph size $N$ grows and the graph becomes sparser according to $q=N^{-\alpha}$,
there exists a critical value of $\alpha = \frac{2}{3}$, below which there exists a computationally intensive procedure that can detect far smaller communities than any computationally efficient procedure,
and above which a linear-time procedure is statistically optimal. The results also lead to the average-case hardness results for recovering the dense community and approximating the densest $K$-subgraph.
\end{abstract}

\section{Introduction}
Networks often exhibit community structure with many edges joining the vertices of the same community and relatively few edges joining vertices of different communities. Detecting communities in networks has received a large amount of attention and has found numerous applications in social and biological sciences, etc (see, \eg, the exposition \cite{Fortunato10} and the references therein).
While most previous work focuses on identifying the vertices in the communities, this paper studies the more basic problem of detecting the presence of a small community in a large random graph, proposed recently in \cite{arias2013community}. This problem has practical applications including detecting new events and monitoring clusters, and is also of theoretical interest for understanding the statistical and algorithmic limits of community detection \cite{ChenXu14}.

Inspired by the model in \cite{arias2013community}, we formulate this community detection problem as a planted dense subgraph detection (\PDS) problem. Specifically, let $\mathcal{G}(N,q)$ denote the Erd\H{o}s-R\'enyi random graph with $N$ vertices, where each pair of vertices is connected independently with probability $q$. Let $\mathcal{G}(N,K,p,q)$ denote the planted dense subgraph model with $N$ vertices where: (1) each vertex is included in the random set $S$ independently with probability $\frac{K}{N}$; (2) for any two vertices, they are connected independently with probability $p$ if both of them are in $S$ and with probability $q$ otherwise, where $p > q$. In this case, the vertices in $S$ form a community with higher connectivity than elsewhere. The planted dense subgraph here has a random size with mean $K$, which is similar to the models adopted in \cite{Decelle11,Mossel12,Mossel13,Feldman13,Massoulie13}, instead of a deterministic size $K$ as assumed in \cite{arias2013community,verzelen2013sparse,ChenXu14}.
\begin{definition}\label{def:HypTesting}
The planted dense subgraph detection problem with parameters $(N,K,p,q)$, henceforth denoted by $\PDS(N,K,p,q)$, refers to the problem of distinguishing hypotheses:
 \begin{align*}
H_0: \quad & G \sim \mathcal{G}(N,q)\triangleq  \Prob_0,  \\
H_1: \quad & G \sim \mathcal{G}(N,K,p,q) \triangleq  \Prob_1.
\end{align*}
\end{definition}

The statistical difficulty of the problem depends on the parameters $(N,K,p,q)$.
Intuitively, if  the expected dense subgraph size $K$ decreases, or if the edge probabilities $p$ and $q$ both decrease
by the same factor, or if $p$ decreases for $q$ fixed, the distributions under the null and alternative hypotheses become less distinguishable.
%get closer and thus the detection becomes harder.
Recent results in \cite{arias2013community,verzelen2013sparse} obtained necessary and sufficient conditions for detecting planted dense subgraphs %threshold for the detection is statistically possible
under certain assumptions of the parameters.
%Moreover, the \PDS($N,K,p,q$)  poses
%computationally challenge: A test exhaustively searching over all the possible subgraph of size $K$ would have a time complexity $O(N^K)$,
%which is computationally intractable for large graphs.
However, it remains unclear whether the statistical fundamental limit can always be achieved by efficient procedures. In fact, it has been
shown in \cite{arias2013community,verzelen2013sparse} that many popular low-complexity tests, such as total degree test, maximal degree test, dense subgraph test, as well as tests based on certain convex relaxations, can be highly suboptimal.
This observation prompts us to investigate the computational limits for the \PDS problem, \ie,
what is the sharp condition on $(N,K,p,q)$ under which the problem admits a computationally efficient test with vanishing error probability, and conversely, without which no algorithm can detect the planted dense subgraph reliably in polynomial time.
%identifying the regime of $(N,K,p,q)$ in which the problem admits a computationally efficient test but unsolvable in polynomial time outside this regime.
To this end, we focus on a particular case where the community is denser by a constant factor than the rest of the graph, \ie, $p=cq$ for some constant $c>1$. %Adopting the standard reduction approach in complexity theory, we  show that the \PDS problem in some parameter regime is at least
%as hard as a well-known computationally intractable problem. In this paper, we use the planted clique detection (\PC) problem as our benchmark.
Adopting the standard reduction approach in complexity theory, we show that the \PDS problem in some parameter regime is at least
as hard as the planted clique problem in some parameter regime, which is conjectured to be computationally intractable.
Let $\calG(n,k,\gamma)$ denote the planted clique model in which we add edges to $k$ vertices uniformly chosen from $\calG(n,\gamma)$ to form a clique.
\begin{definition}\label{def:PlantedCliqueDetection}
The \PC detection problem with parameters $(n,k,\gamma)$, denoted by $\PC(n,k,\gamma)$ henceforth, refers to the problem of distinguishing hypotheses:
 \begin{align*}
H^{\rm C}_0: \quad & G \sim \mathcal{G}(n,\gamma),  \\
H^{\rm C}_1: \quad & G \sim \mathcal{G}(n,k,\gamma).
\end{align*}
\end{definition}

The problem of finding the planted clique has been extensively studied for $\gamma=\frac{1}{2}$ and the state-of-the-art polynomial-time algorithms \cite{Alon98,Feige99findingand,McSherry01,Feige10findinghidden,Dekel10,ames2011plantedclique,Deshpande12} only work for $k=\Omega(\sqrt{n})$. There is no known polynomial-time solver for the \PC  problem for $k=o(\sqrt{n})$ and any constant $\gamma>0$. It is conjectured \cite{Jer92,Hazan2011Nash,Juel00cliqueCrypto,alon2007testing,Feldman13} that the \PC problem cannot be solved in polynomial time for $k=o(\sqrt{n})$ with $\gamma =\frac{1}{2}$, which we refer to as the \PC Hypothesis.
\begin{hypothesis}\label{hyp:HypothesisPlantedClique}
Fix some constant $0<\gamma \le \frac{1}{2}$. For any sequence of randomized polynomial-time tests $\{ \psi_{n,k_n} \}$ such that $\limsup_{n \to \infty} \frac{\log k_n }{ \log n} < 1/2$,
\begin{align*}
\liminf_{n \to \infty} \mathbb{P}_{H_0^{\rm C}} \{ \psi_{n,k} (G) =1 \} + \mathbb{P}_{H_1^{\rm C}} \{ \psi_{n,k}(G)=0 \} \ge 1.
\end{align*}
\end{hypothesis}
The \PC Hypothesis with $\gamma=\frac{1}{2}$ is similar to \cite[Hypothesis 1]{MaWu13} and \cite[Hypothesis $\mathbf{B_{PC}}$]{berthet2013lowerSparsePCA}.
Our computational lower bounds require that the \PC Hypothesis holds for any positive constant $\gamma$.
%, which is still believed to be true.
% Indeed, \cite[Theorem 10.3]{ABW10} justify a hardness assumption on constructing the public key by assuming the \PC  Hypothesis holds for $\gamma=2^{-\log^{0.99} n}$.
An even stronger assumption that \PC Hypothesis holds for $\gamma=2^{-\log^{0.99} n}$ has been used in \cite[Theorem 10.3]{ABW10} for public-key cryptography.
 Furthermore, \cite[Corollary 5.8]{Feldman13} shows that under a statistical query model, any statistical algorithm requires at least $n^{\Omega(\frac{\log n}{ \log (1/\gamma) } ) }$ queries for detecting the planted bi-clique in an Erd\H{o}s-R\'enyi random bipartite graph with edge probability $\gamma$.

\subsection{Main Results}
We consider the $\PDS(N,K,p,q)$ problem in the following asymptotic regime:
\begin{equation}
p=cq= \Theta(N^{-\alpha}),\; K =\Theta(N^{\beta}), \quad N \diverge,
	\label{eq:scaling}
\end{equation}
where $c>1$ is a fixed constant, $\alpha \in [0,2]$ governs the sparsity of the graph,\footnote{The case of $\alpha>2$ is not interesting since detection is impossible even if the planted subgraph is the entire graph ($K=N$).} and $\beta \in [0,1]$ captures the size of the dense subgraph.
%under the scaling of $p=cq= \Theta(N^{-\alpha})$, $K =\Theta(N^{\beta})$ and $N \diverge$ for $ \alpha>0$ and $\beta \in (0,1)$ and a fixed constant $c>1$. The sparsity of the graph and the size of the dense subgraph are captured by the parameter $\alpha$ and $\beta$, respectively.
Clearly the detection problem becomes more difficult if either $\alpha$ increases or $\beta$ decreases. Assuming the \PC Hypothesis holds for any positive constant $\gamma$, we show that the parameter space of $(\alpha,\beta)$ is partitioned into three regimes as depicted in \prettyref{fig:phase}:

\begin{compactitem}
\item \textbf{The Simple Regime: $ \beta> \frac{1}{2}+\frac{\alpha}{4} $}. The dense subgraph can be detected in linear time with high probability by thresholding the total number of edges.
\item \textbf{The Hard Regime: $\alpha<\beta <\frac{1}{2}+\frac{\alpha}{4} $}. Reliable detection can be achieved by thresholding the maximum number of edges among all subgraphs of size $K$; however, no polynomial-time solver exists in this regime.
%the densest $K$-subgraphs
%A super polynomial time scan test which searches over all possible subgraphs of size $K$ succeeds in this regime, but no polynomial-time solver exists in this regime.
% if the \PC Hypothesis holds for any constant $\gamma$.
\item \textbf{The Impossible Regime: $ \beta< \min \{\alpha, \frac{1}{2}+\frac{\alpha}{4} \} $}. No test can detect the planted subgraph regardless of the computational complexity.
\end{compactitem}
\begin{figure}[h!]
  \begin{center}
\scalebox{1}{
\begin{tikzpicture}[scale = 2.5, font = \small, thick]
\draw (4, 0) -- (4, 2);
\draw (0, 2) -- (4, 2);
\node [left] at (0, 2) {$1$};
\node [below] at (4, 0) {$2$};
\node [above] at (4.7,0) {$p=cq=\Theta(N^{-\alpha})$};
\node [right] at (0,2.3) {$K=\Theta(N^{\beta})$};
\node [left] at (0, 1) {$1/2$};
%\node [right] at (2,1.5){$3/4$};
\draw[green] (0,1)--(0,2);
\draw[red] (0,0)--(0,1);
\path[fill = black!20] (0, 0) -- (4,0) -- (4, 2) -- (4/3,4/3)-- cycle;
\node at (2,0.5){impossible};
\path[fill= red!60] (0,0) --(2,2)-- (0,2)-- cycle;
\node at (0.3,0.7) {hard};
\node [below] at (0.6,0.6) [rotate=45] {$\beta=\alpha$};
\path[fill=green!60] (0,1)--(4,2)--(0,2)-- cycle;
\node at (1.7,1.7){simple};
\node [left] at (0, 1) {$1/2$};
%\node at (1.8,1.3) {\small spectral barrier};
%\draw[->] (1.8,1.4) -- (1.5,1.75);
%node [above] at (1.2,0.9) {\small computational limit};
%\draw[->] (1.2,1.2) -- (1,1.5);
%\draw[thin, dashed] (4/3, 4/3) -- (2, 2);
%\draw[thin, dashed] (0,1)--(1,1);
\draw[->] (0, 0) node [below left] {$O$}-- (4.5, 0) node [below]{\large $\alpha$};
\draw[->] (0, 0) -- (0, 2.3) node [left]{\large $\beta$};
%\draw[thin,dashed] (4/3,4/3) -- (2,2);
%\draw[thin,dashed] (2,0)--(2,2);
\node [below] at (2,0) {$1$};
\node [below] at (4/3,0) {$2/3$};
%\node [left] at (0,4/3) {$2/3$};
\node [below] at (2,3/2) [rotate=14]{$ \beta= \alpha/4+ 1/2$ };
\end{tikzpicture}}
\end{center}
%\caption{The simple (green), hard (red), impossible (gray) regimes for the \PDS problem.}
\caption{The simple (green), hard (red), impossible (gray) regimes for detecting the planted dense subgraph.}
\label{fig:phase}
\end{figure}
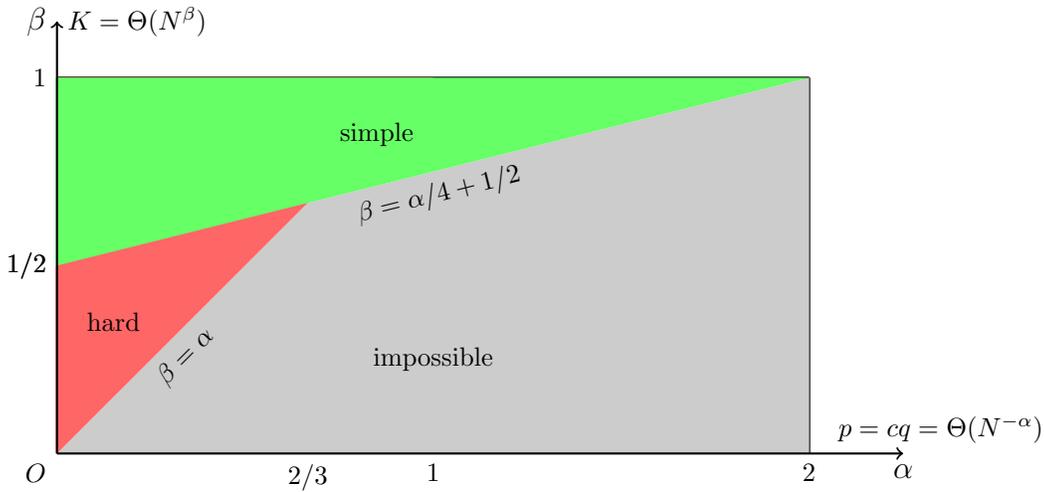

The computational hardness of the \PDS problem exhibits a phase transition at the critical value $\alpha=2/3$: For {\it moderately sparse} graphs with $\alpha<2/3$, there exists a combinatorial algorithm that can detect far smaller communities than any efficient procedures; For \emph{highly sparse} graphs with $\alpha>2/3$, optimal detection is achieved in linear time based on the total number of edges. Equivalently, attaining the statistical detection limit is computationally tractable only in the large-community regime ($\beta > 2/3$).
%if one restricts to efficiently computable tests, small communities
%Putting it in an another equivalent way, the computational hardness exhibits phase transition at $\beta=2/3$: In the {\it small} community regime with $\beta< 2/3$, a computational expensive test succeeds in the much sparser graph than any computationally efficient test; In the \emph{large} community regime with $\beta>2/3$, optimal detection is achieved in linear time based on the total number of edges.
Therefore, surprisingly, the linear-time test based on the total number of edges is always statistically optimal among all computationally efficient procedures in the sense that no polynomial-time algorithm can reliably detect the community when $\beta < \frac{1}{2}+\frac{\alpha}{4}$.
%it attains the best statistical performance.
It should be noted that \prettyref{fig:phase} only captures the leading polynomial term according to the parametrization \prettyref{eq:scaling}; at the boundary $\beta=\alpha/4+1/2$, it is plausible that one needs to go beyond simple edge counting in order to achieve reliable detection. This is analogous to the planted clique problem where the maximal degree test succeeds if the clique size satisfies $k=\Omega(\sqrt{n \log n})$ \cite{Kucera95} and the more sophisticated spectral method succeeds if $k=\Omega(\sqrt{n})$ \cite{Alon98}.

The above hardness result should be contrasted with the recent study of community detection on the stochastic block model, where the community size scales linearly with the network size. When the edge density scales as $\Theta(\frac{1}{N})$ \cite{Mossel12,Mossel13,Massoulie13} (resp. $\Theta(\frac{\log N}{N})$ \cite{Abbe14,Mossel14,HWX14b}), the statistically optimal threshold for partial (resp. exact) recovery can be attained in polynomial time up to the sharp constants.
In comparison, this paper focuses on the regime when the community size grows \emph{sublinearly} as $N^{\beta}$ and the edge density decays more slowly as $N^{-\alpha}$. It turns out that in this case even achieving the optimal exponent is computationally as demanding as solving the planted clique problem.

Our computational lower bound for the \PDS problem also implies the average-case hardness of approximating the planted dense subgraph or the densest $K$-subgraph of the random graph ensemble $\calG(N,K,p,q)$, complementing the worst-case inapproximability result in \cite{Alon11}, which is based on the planted clique hardness as well. In particular, we show that no polynomial-time algorithm can approximate the planted dense subgraph or the densest $K$-subgraph within any constant factor in the regime of $\alpha<\beta <\frac{1}{2}+\frac{\alpha}{4}$, which provides a partial answer to the conjecture made in \cite[Conjecture 2.6]{ChenXu14} and the open problem raised in \cite[Section 4]{Alon11} (see \prettyref{sec:recovery}). Our approach and results can be extended to the bipartite graph case (see \prettyref{sec:bipartite}) and shed light on the computational limits of the \PDS problem with a fixed planted dense subgraph size studied in \cite{arias2013community,verzelen2013sparse} (see \prettyref{sec:fixedsize}).
%\nb{need REWORD: does not exactly extend to fixed size. otherwise why keep distributional.}

\subsection{Connections to the Literature}
This work is inspired by an emerging line of research (see, \eg,
%that take a joint statistical and computational view on statistical inference problems
 \cite{Kolar2011,balakrishnan2011tradeoff,berthet2013lowerSparsePCA,chandrasekaran2013tradeoff,MaWu13,ChenXu14,Jiaming13}) which examines high-dimensional inference problems from both the statistical and computational perspectives. Our computational lower bounds follow from a randomized polynomial-time reduction scheme which approximately reduces the \PC problem to the \PDS problem of appropriately chosen parameters. Below we discuss the connections to previous results and highlight the main technical contributions of this paper.

\paragraph*{\PC Hypothesis} Various hardness results in the theoretical computer science
literature have been established based on the \PC  Hypothesis with $\gamma=\frac{1}{2}$, \eg cryptographic applications \cite{Juel00cliqueCrypto}, approximating Nash equilibrium \cite{Hazan2011Nash}, testing $k$-wise independence \cite{alon2007testing}, etc.
% approximating densest $k$-subgraph \cite{Alon11}, certifying the Restricted Isometry Property for compressed sensing measurement matrices \cite{Pascal12}, etc.
 More recently, the \PC  Hypothesis with $\gamma=\frac{1}{2}$ has been used to investigate the penalty incurred by complexity constraints on certain high-dimensional statistical inference problems, such as
%  to the existence of big gap between statistical and computational limits in
  detecting sparse principal components \cite{berthet2013lowerSparsePCA} and noisy biclustering (submatrix detection) \cite{MaWu13}.  Compared with most previous works, our computational lower bounds rely on the stronger assumption that the \PC  Hypothesis holds for any positive constant $\gamma$.  An even stronger assumption that \PC Hypothesis holds for $\gamma=2^{-\log^{0.99} n}$ has been used in \cite{ABW10} for public-key cryptography. It is an interesting open problem to prove that \PC  Hypothesis for a fixed $\gamma \in (0,\frac{1}{2})$ follows from that for $\gamma=\frac{1}{2}$.
%Planted Clique Hypothesis holds in the class of statistical algorithms for any $\gamma$ so long as $\log \frac{1}{\gamma} =o(\log n)$.  \nb{--this is from \cite[Corollary 5.8]{Feldman13} for the class of ``statistical algorithms'' and ``query complexity'' no?}

\paragraph*{Reduction from the \PC Problem}
Most previous work \cite{Hazan2011Nash,alon2007testing,Alon11,ABW10} in the theoretical computer science
literature uses the reduction from the \PC  problem to generate computationally hard instances of problems and establish \emph{worst-case} hardness results; the underlying distributions of the instances could be arbitrary. Similarly, in the recent works \cite{berthet2013lowerSparsePCA,MaWu13} on the computational limits of certain \emph{minimax} inference problems, the reduction from the \PC  problem is used to generate computationally hard but statistically feasible instances of their problems; the underlying distributions of the instances can also be arbitrary as long as they are valid priors on the parameter spaces.
%the statistical limits remain the same.
%In particular, instead of focusing on the original sparse principal component detection problem, i.e., testing the standard Gaussian distribution versus the standard Gaussian distribution plus a sparse principal component, \cite{berthet2013lowerSparsePCA} considered testing sub-Gaussian distributions versus sub-Gaussian distributions plus a sparse principal component.
%Likewise, instead of focusing on the original submatrix detection problem, i.e., testing the standard Gaussian random matrix versus
% the standard Gaussian random matrix plus a constant-valued submatrix located uniformly at random, \cite{MaWu13} considered testing
% the standard Gaussian random matrix versus the standard Gaussian random matrix plus a submatrix located arbitrarily.
 In contrast, here our goal is to establish the average-case hardness of the \PDS problem based on that of the \PC problem.
 %   showing computational lower bounds for a specific hypothesis testing problem.
   Thus the underlying distributions of the problem instances generated from the reduction must be close to the desired distributions in total variation under both the null and alternative hypotheses. To this end, we start with a small dense graph generated from $\calG(n,\gamma)$ under $H_0$ and $\calG(n,k,\gamma)$ under $H_1$, and arrive at a large sparse graph whose distribution is exactly $\calG(N,q)$ under $H_0$ and approximately equal to $\calG(N,K,p,q)$ under $H_1$.
   Notice that simply sparsifying the \PC problem does not capture the desired tradeoff between the graph sparsity and the cluster size.
   Our reduction scheme differs from those used in \cite{berthet2013lowerSparsePCA,MaWu13} which start with a large dense graph. Similar to ours, the reduction scheme in \cite{Alon11} also enlarges and sparsifies the graph by taking its subset power; but the distributions of the resulting random graphs are rather complicated and not close to the Erd\H{o}s-R\'enyi type.
%   This is exactly why our average case inapproximability result to densest $K$-subgraph problem is considerably weaker than the inapproximability result obtained in \cite{Alon11}, which we elaborate on in the next paragraph.

\paragraph*{Inapproximability of the \DKS Problem}
The densest $K$-subgraph (\DKS) problem refers to finding the subgraph of $K$ vertices with the maximal number of edges.
In view of the NP-hardness of the \DKS problem which follows from the NP-hardness of MAXCLIQUE, it is of interest to consider an $\eta$-factor approximation algorithm, which outputs a subgraph with $K$ vertices containing at least a $\frac{1}{\eta}$-fraction of the number of edges in the densest $K$-subgraph.
%This problem is NP-hard which follows from the NP-hardness of MAXCLIQUE. An $\eta$-factor approximation algorithm for \DKS problem outputs a subgraph with $K$ vertices that contains $\frac{1}{\eta}$ times as many edges as the densest $K$-subgraph.
Proving the NP-hardness of $(1+\epsilon)$-approximation for \DKS for any fixed $\epsilon>0$ is a longstanding open problem. See \cite{Alon11} for a comprehensive discussion. Assuming the \PC Hypothesis holds with $\gamma=\frac{1}{2}$, \cite{Alon11} shows that the \DKS problem is hard to approximate within any constant factor even if the densest $K$-subgraph is a clique of size $K=N^{\beta}$ for any $\beta<1$, where $N$ denotes the total number of vertices. This worst-case inapproximability result is in stark contrast to the average-case behavior in the
planted dense subgraph model $G(N,K,p,q)$ under the scaling \prettyref{eq:scaling}, where it is known \cite{ChenXu14,ames2013robust} that the planted dense subgraph can be exactly recovered in polynomial time if $\beta> \frac{1}{2} + \frac{\alpha}{2}$ (see the simple region in Fig.~\ref{fig:planteddensesubgraph} below),  implying that the densest $K$-subgraph can be approximated within a factor of $1+\epsilon$ in polynomial time for any $\epsilon>0$. On the other hand, our computational lower bound for $\PDS(N,K,p,q)$ shows that any constant-factor approximation of the densest $K$-subgraph has high average-case hardness if $\alpha<\beta < \frac{1}{2}+\frac{\alpha}{4}$ (see \prettyref{sec:recovery}).

%\paragraph*{Variants of \PDS Model}
%Various models for planted dense subgraphs have been studied in the literature.
%Our \PDS model is closely related to the one with a fixed size $K$ studied in \cite{arias2013community,verzelen2013sparse} (see \prettyref{sec:fixedsize}); the only difference is that here the size of the planted dense subgraph is binomially distributed with mean $K$.
%%as $\Binom(N,K/N)$ instead of deterministic $K$.
%Three versions of the \PDS model in increasing order of difficulty have been considered in \cite[Section 3]{BCCFV10}: The \emph{random planted} model is the same as \PDS model in \cite{arias2013community,verzelen2013sparse} except for different parameterizations of edge probabilities $p,q$; For the \emph{dense in random} model, the planted dense $K$-subgraph could be arbitrary instead of the Erd\H{o}s-R\'enyi random graph assumed in the random planted model; For the \emph{dense versus random} model, the graph under the alternative hypothesis could be an arbitrary graph which contains a dense $K$-subgraph.  A bipartite graph variant of the \PDS model is used in \cite[p.\ 10]{ABBG10} for financial applications where the total number of edges is the same under both the null and alternative hypothesis.
%%which is harder to solve \nb{why???} than our model because there the graph is regular which prevents the linear test by thresholding the total number of edges from succeeding.
%A hypergraph variant of the \PDS problem is used in \cite{ABW10} for cryptographic applications.

\paragraph*{Variants of \PDS Model}
Three versions of the  \PDS model were considered in \cite[Section 3]{BCCFV10}.   Under all three the graph under the null hypothesis is the
 Erd\H{o}s-R\'enyi graph.   The versions of the alternative hypothesis, in order of increasing difficulty of detection, are:
(1)  The \emph{random planted} model, such that the graph under the alternative hypothesis is obtained by
generating an  Erd\H{o}s-R\'enyi graph, selecting $K$ nodes arbitrarily, and then resampling the edges among the $K$ nodes with a higher probability
to form a denser Erd\H{o}s-R\'enyi subgraph.   This is somewhat more difficult to detect than the model of \cite{arias2013community,verzelen2013sparse},
for which the choice of which $K$ nodes are in the planted dense subgraph is made before any edges of the graph are independently, randomly generated.
(2) The \emph{dense in random} model,  such that both the nodes and edges of the planted dense $K$-subgraph are arbitrary;
(3) The \emph{dense versus random} model, such that the entire graph under the alternative hypothesis could be an arbitrary graph containing a dense $K$-subgraph.
Our \PDS model is closely related to the first of these three versions, the key difference being that for our model the size of the planted dense subgraph is binomially distributed with mean $K$ (see Section \ref{sec:fixedsize}).   Thus, our hardness result is for the easiest type
of detection problem.     A bipartite graph variant of the \PDS model is used in \cite[p.\ 10]{ABBG10} for financial applications where the total number of edges is the same under both the null and alternative hypothesis.
%which is harder to solve \nb{why???} than our model because there the graph is regular which prevents the linear test by thresholding the total number of edges from succeeding.
A hypergraph variant of the \PDS problem is used in \cite{ABW10} for cryptographic applications.

\subsection{Notations}
For any set $S$, let $|S|$ denote its cardinality. Let $s_1^n=\{s_1, \ldots, s_n \}$.
For any positive integer $N$, let $[N]=\{1, \ldots, N\}$. For $a, b \in \reals$, let $a \wedge b =\min\{a, b\}$ and $a \vee b =\max \{a, b\}$.  We use standard big $O$ notations, e.g., for any sequences $\{a_n\}$ and $\{b_n\}$, $a_n=\Theta(b_n)$ if there is an absolute constant $C>0$ such that $1/C \le a_n/ b_n \le C$.
%In addition, we write $a_n=\breve{\Theta} (b_n)$ if they are equivalent up to sub-polynomial factors, i.e., for any $\delta>0$, $n^{-\delta} \le a_n/b_n \le n^{\delta}$ holds for sufficiently large values of $n$.
Let $\Bern(p)$ denote the Bernoulli distribution with mean $p$ and $\Binom(N,p)$ denote the binomial distribution with $N$ trials and success probability $p$.
For random variables $X, Y$, we write $X \Perp Y$ if $X$ is independent with $Y$.
For probability measures $\mathbb{P}$ and $\mathbb{Q}$, let $d_{\rm TV}(\mathbb{P},\mathbb{Q} ) = \frac{1}{2} \int |\diff \Prob-\diff \mathbb{Q}|$ denote the total variation distance and
%$\chi^2(\mathbb{P}\|\mathbb{Q}) =  \frac{1}{2} \int (\frac{\diff \Prob}{\diff \mathbb{Q}}-1)^2\diff \mathbb{Q}$
$\chi^2(\mathbb{P}\|\mathbb{Q}) =  \int \frac{(\diff \Prob-\diff \mathbb{Q})^2}{\diff \mathbb{Q}}$
 the $\chi^2$-divergence.
The distribution of a random variable $X$ is denoted by $P_X$. We write $X\sim \Prob$ if $P_X=\Prob$.
%Throughout the paper, let $\Prob_0$ (resp.\ $\Prob_1$) denote the distribution of the adjacency matrix $A$ of graph $G$ under $H_0$ (resp. $H_1$) of the \PC problem.
%distributed as $\Prob$ is denoted by $X \sim \Prob$.
All logarithms are natural unless the base is explicitly specified.

\section{Statistical Limits} \label{sec:statisticallimits}
This section determines the statistical limit for the $\PDS(N,K,p,q)$ problem with $p=cq$ for a fixed constant $c>1$. For a given pair $(N,K)$, one can ask the question: What is the smallest density $q$
%, say $q^*$,
such that it is possible to reliably detect the planted dense subgraph?
When the subgraph size $K$ is \emph{deterministic}, this question has been thoroughly investigated by Arias-Castro and Verzelen \cite{arias2013community,verzelen2013sparse} for general $(N,K,p,q)$ and the statistical limit with sharp constants has obtained in certain asymptotic regime.  Their analysis treats the dense regime $\log (1 \vee (Kq)^{-1} )=o(\log \frac{N}{K})$ \cite{arias2013community} and sparse regime $\log \frac{N}{K} =O(\log (1 \vee (Kq)^{-1} ))$ \cite{verzelen2013sparse} separately. Here as we focus on the special case of $p=cq$ and are only interested in characterizations within absolute constants, we provide a simple non-asymptotic analysis which treats the dense and sparse regimes in a unified manner. Our results demonstrate that the $\PDS$ problem in \prettyref{def:HypTesting}
has the same statistical detection limit as the $\PDS$ problem with a deterministic size $K$ studied in \cite{arias2013community,verzelen2013sparse}.
%our non-asymptotic analysis is simple and does not need to treat two regimes separately.
%
%\nb{todo: reword since here the model is random size}

\subsection{Lower Bound}
By the definition of the total variation distance, the optimal testing error probability is determined by the total variation distance between the distributions under the null and the alternative hypotheses:
\begin{align*}
\min_{\phi: \{0,1\}^{N(N-1)/2}\to \{0,1\} } \left( \Prob_0\{\phi(G)=1\} + \Prob_1\{\phi(G)=0\} \right) = 1- d_{\rm TV} (\Prob_0, \Prob_1 ).
\end{align*}
The following result (proved in \prettyref{sec:pf-lb}) shows that if $q = O(\frac{1}{K} \log \frac{eN}{K} \wedge \frac{N^2}{K^4})$, then there exists no test which can detect the planted subgraph reliably.
%upper bounds the $d_{\rm TV} (\Prob_0, \Prob_1 )$ and thus establishes the statistical lower bound.
%\begin{lemma} \label{prop:lowerbound}
%Suppose $p=cq$ for a constant $c>1$. There exists a constant $C$ which only depends on $c$ such that for any $1\le K\le N$ and any $q \le C( \frac{1}{K} \log \frac{eN}{K} \wedge \frac{N^2}{K^4} )$,
%\begin{align*}
%d_{\rm TV} (\Prob_0, \Prob_1 ) \le 1/2.
%\end{align*}
%\end{lemma}
\begin{proposition} \label{prop:lowerbound}
Suppose $p=cq$ for some constant $c>1$. There exists a function $h:\reals_+ \to \reals_+$
%depending on $c$ and
satisfying $h(0+)=0$ such that the following holds: For any $1\le K\le N$, $C>0$ and $q \le C( \frac{1}{K} \log \frac{eN}{K} \wedge \frac{N^2}{K^4} )$,
\begin{align}
d_{\rm TV} (\Prob_0, \Prob_1 ) \le h(Cc^2)+\exp(-K/8).
\label{eq:lb-tv}
\end{align}
\end{proposition}
%\begin{lemma} \label{prop:lowerbound}
%Suppose $p=cq$ for some constant $c>1$.  There exists a universal constant $C>0$ such that for any $1\le K \le N$ and $p \le C( \frac{1}{K} \log \frac{eN}{K} \wedge \frac{N^2}{K^4} )$,
%\begin{align*}
%d_{\rm TV} (\Prob_0, \Prob_1 ) \le 1/4 +\exp(-K/8).
%\end{align*}
%\end{lemma}

\subsection{Upper Bound}
Let $A$ denote the adjacency matrix of the graph $G$. The detection limit can be achieved by
%This subsection gives a non-asymptotic analysis of
the linear test statistic and scan test statistic proposed in \cite{arias2013community,verzelen2013sparse}:
\begin{align}
T_{\rm lin}  \triangleq \sum_{i<j} A_{ij}, \quad T_{\rm scan} \triangleq \max_{S': |S'| = K} \sum_{i,j \in S': i<j} A_{ij}, \label{eq:tests}
\end{align}
which correspond to the total number of edges in the whole graph and the densest $K$-subgraph, respectively. Interestingly, the exact counterparts of these tests have been proposed and shown to be minimax optimal for detecting submatrices in Gaussian noise \cite{butucea2013,Kolar2011,MaWu13}. The following lemma bounds the error probabilities of the linear and scan test.
\begin{proposition} \label{prop:upperbound}
Suppose $p=cq$ for a constant $c>1$. For the linear test statistic, set $\tau_1=\binom{N}{2}q + \binom{K}{2}(p-q)/2$. For the scan test statistic, set $\tau_2=\binom{K}{2}(p+q)/2$. Then there exists a constant $C$ which only depends on $c$ such that
\begin{align*}
\mathbb{P}_0 [ T_{\rm lin} > \tau_1] + \mathbb{P}_1 [ T_{\rm lin} \le \tau_1 ]  &\le 2 \exp \left( - C \frac{K^4q}{N^2}\right) + \exp \left( -\frac{K}{200}\right) \\
\mathbb{P}_0 [ T_{\rm scan} > \tau_2] + \mathbb{P}_1 [ T_{\rm scan} \le \tau_2 ] & \le 2 \exp \left( K \log \frac{Ne}{K} - C K^2q \right) +\exp \left( -\frac{K}{200}\right) .
\end{align*}
\end{proposition}

To illustrate the implications of the above lower and upper bounds, consider the $\PDS(N,K,p,q)$ problem with the parametrization $p=cq$, $q =N^{-\alpha}$ and $K=N^\beta$ for $ \alpha>0$ and $\beta \in (0,1)$ and $c>1$. In this asymptotic regime, the fundamental detection limit is characterized by the following function
\begin{align}
 \beta^\ast(\alpha) \triangleq \alpha \wedge \left( \frac{1}{2} + \frac{\alpha}{4} \right), \label{eq:detectionlimit}
\end{align}
which gives the statistical boundary in \prettyref{fig:phase}.
%\begin{enumerate}
%	\item If $\beta < \beta^*(\alpha)$, then $\dTV \to 0$ as a  consequence of \prettyref{prop:lowerbound};
%	\item If $\beta > \beta^*(\alpha)$, then  \prettyref{prop:upperbound} implies that the test $\phi(A) = \indc{T_{\rm lin} > \tau_1 \text{ or } T_{\rm scan} > \tau_2}$ achieves vanishing Type-I+II error probability.
%\end{enumerate}
Indeed, if $\beta < \beta^*(\alpha)$, as a consequence of \prettyref{prop:lowerbound}, $\Prob_0\{\phi(G)=1\} + \Prob_1\{\phi(G)=0\} \to 1$ for any sequence of tests.
%then $\dTV(\Prob_0,\Prob_1) \to 0$ as a  consequence of \prettyref{prop:lowerbound};
 Conversely, if $\beta > \beta^*(\alpha)$, then  \prettyref{prop:upperbound} implies that the test $\phi(G) = \indc{T_{\rm lin} > \tau_1 \text{ or } T_{\rm scan} > \tau_2}$ achieves vanishing Type-I+II error probabilities. More precisely, the linear test succeeds in the regime $\beta>\frac{1}{2} + \frac{\alpha}{4}$, while the scan test succeeds in the regime $\beta> \alpha$.

Note that $T_{\rm lin}$ can be computed in linear time. However, computing $T_{\rm scan}$ amounts to enumerating all subsets of $[N]$ of cardinality $K$, which can be computationally intensive.
%takes at least $\Theta(N^K)$ steps.
Therefore it is unclear whether there exists a polynomial-time solver in the regime $\alpha< \beta <\frac{1}{2} + \frac{\alpha}{4}$.
Assuming the \PC Hypothesis, this question is resolved in the negative in the next section.

\section{Computational Lower Bounds}\label{sec:computationallimits}
In this section, we establish the computational lower bounds for the \PDS problem assuming the intractability of the planted clique problem. %To this end, we first formally define the planted clique hypothesis. Then,
We show that the \PDS problem can be approximately reduced from the \PC problem of appropriately chosen parameters in randomized polynomial time.
%randomized polynomial-time reduction scheme which reduces the planted clique detection problem to the distributional planted densest subgraph detection problem.
Based on this reduction scheme, we establish a formal connection between the \PC problem and the \PDS problem in \prettyref{prop:reduction}, and the desired computational lower bounds follow as \prettyref{thm:main}.

We aim to reduce the $\PC(n,k,\gamma)$ problem to the $\PDS(N,K,cq,q)$ problem.
% with $p=2q$ via the following randomized polynomial-time reduction.
For simplicity, we focus on the case of $c=2$; the general case follows similarly with a change
in some numerical constants that come up in the proof.
% Fix a positive integer $\ell$ and let $N=n\ell$, $K=k\ell$.
We are given an adjacency matrix $A \in \{0,1\}^{n \times n}$, or equivalently, a graph $G,$  and with the help of additional
randomness, will map it to an adjacency matrix $\tA \in \{0,1\}^{N \times N},$ or equivalently, a graph $\tG$
such that the hypothesis $H_0^{\rm C}$ (resp.\ $H_1^{\rm C}$) in \prettyref{def:PlantedCliqueDetection} is mapped to $H_0$ exactly (resp.\ $H_1$ approximately) in \prettyref{def:HypTesting}. In other words, if $A$ is drawn from $\calG(n,\gamma)$, then $\tA$ is distributed according to $\Prob_0$; If $A$ is drawn from $\calG(n,k,1,\gamma)$, then the distribution of $\tA$ is close in total variation to $\Prob_1$.
%%In other words, if $A$ is drawn from $\calG(n,\gamma)$ under $H_0^{\rm C}$, then $\tA$ is distributed according to the null distribution $\Prob_0$ under $H_0$; If $A$ is drawn from $\calG(n,k,1,\gamma)$ under $H_1^{\rm C}$, then the distribution of $\tA$ is close in total variation distance to $\Prob'_1$ under $H'_1$.

Our reduction scheme works as follows. Each vertex in $\tG$ is randomly assigned a parent vertex in $G,$
with the choice of parent being made independently for different vertices in $\tG,$  and uniformly
over the set $[n]$ of vertices in $G.$  Let $V_s$ denote the set of vertices in $\tG$ with parent
$s\in [n]$  and let $\ell_s=|V_s|$. Then the set of children nodes $\{V_s: s \in [n]\}$ form a random partition of $[N]$.
For any $1 \leq s \leq t \leq n,$ the number of edges, $E(V_s,V_t)$, from vertices in $V_s$ to vertices in $V_t$
in $\tG$ will be selected randomly with a conditional probability distribution specified below.
Given  $E(V_s,V_t),$  the particular  set of edges with cardinality $E(V_s,V_t)$ is chosen uniformly at
random.
% from the set of all possibilities.

It remains to specify, for $1\leq s \leq t \leq n,$
the conditional distribution of $E(s,t)$ given $l_s, l_t,$ and $A_{s,t}.$
Ideally, conditioned on $\ell_s$ and $\ell_t$, we want to construct
% the number of edges $E(s,t)$ to be distributed according to $\Binom()$
a Markov kernel from $A_{s,t}$ to $E(s,t)$ which maps  $\Bern(1)$  to the desired edge distribution $\Binom(\ell_s\ell_t,p)$, and $\Bern(\gamma)$  to $\Binom(\ell_s\ell_t,q)$, depending on whether both $s$ and $t$ are in the clique or not, respectively. Such a kernel, unfortunately, provably does not exist. Nonetheless, this objective can be accomplished approximately in terms of the total variation. For $s=t \in [n],$  let $E(V_s,V_t)\sim\Binom( \binom{\ell_t}{2}, q).$
For $ 1 \le s < t \le n$, denote $P_{\ell_s \ell_t} \triangleq \Binom(\ell_s \ell_t,p)$ and $Q_{\ell_s \ell_t} \triangleq \Binom(\ell_s \ell_t,q)$.
Fix $0 < \gamma \leq \frac{1}{2}$ and put $m_0 \triangleq \lfloor \log_2 (1/\gamma) \rfloor$.
Define
    \begin{align*}
    P'_{\ell_s \ell_t} (m)=  \left\{
    \begin{array}{rl}
    P_{\ell_s \ell_t} (m) + a_{\ell_s \ell_t} & \text{for } m=0,\\
    P_{\ell_s \ell_t}(m) & \text{for } 1 \le m \le m_0, \\
    \frac{1}{\gamma} Q_{\ell_s \ell_t} (m) & \text{for } m_0 < m \le \ell_s \ell_t .
    \end{array} \right.
    \end{align*}
    where $a_{\ell_s \ell_t}=\sum_{m_0<m \le  \ell_s\ell_t} [ P_{\ell_s \ell_t}(m) - \frac{1}{\gamma} Q_{\ell_s \ell_t}(m) ]$.
    Let $Q'_{\ell_s \ell_t} = \frac{1}{1-\gamma} (Q_{\ell_s \ell_t} - \gamma P'_{\ell_s \ell_t})$.
%    and $\widetilde{Q}_{\ell_s \ell_t}$ is defined such that $(1-\gamma) \widetilde{Q}_{\ell_s \ell_t}+ \gamma \widetilde{P}_{\ell_s \ell_t} = Q_{\ell_s \ell_t}$.
As we show later, $Q'_{\ell_s \ell_t}$ and $P'_{\ell_s \ell_t}$ are well-defined probability distributions as long as $\ell_s, \ell_t \le 2 \ell$ and $16 q \ell^2 \le 1$, where $\ell=N/n$.  Then, for $1\leq s < t \leq n,$ let the conditional distribution of $E(V_s,V_t)$ given $\ell_s,\ell_t,$ and $A_{s,t}$ be given by
    \begin{align}
    E(V_s, V_t) \sim \left\{
    \begin{array}{rl}
    P'_{\ell_s \ell_t} & \text{if } A_{st}=1, \ell_s, \ell_t \le 2\ell\\
    Q'_{\ell_s \ell_t} & \text{if } A_{st}=0, \ell_s, \ell_t \le 2 \ell \\
    Q_{\ell_s \ell_t} & \text{if } \max\{ \ell_s, \ell_t \}> 2 \ell.
    \end{array} \right. \label{eq:edgedist}
    \end{align}

The next proposition (proved in \prettyref{sec:pf-reduction}) shows that the randomized reduction defined above maps $\calG(n,\gamma)$ into $\calG(N,q)$ under the null hypothesis and $\calG(n,k,\gamma)$ approximately  into $\calG(N,K,p,q)$ under the alternative hypothesis, respectively.
The intuition behind the reduction scheme is as follows:
By construction, $(1-\gamma) Q'_{\ell_s \ell_t}+ \gamma P'_{\ell_s \ell_t} = Q_{\ell_s \ell_t}= \Binom(\ell_s \ell_t,q)$ and therefore the null distribution of the \PC problem is exactly matched to that of the \PDS problem, i.e., $P_{\tG|H_0^C}=\Prob_0$.
%It remains to show that the alternative distributions are approximately matched, which constitutes the core of the proof.
The core of the proof lies in establishing that the alternative distributions are approximately matched.
The key observation is that $P'_{\ell_s \ell_t}$ is close to $P_{\ell_s \ell_t}= \Binom(\ell_s \ell_t,p)$ and thus for nodes with distinct parents $s \neq t$ in the planted clique, the number of edges $E(V_s,V_t)$ is approximately distributed as the desired $\Binom(\ell_s \ell_t,p)$; for nodes with the same parent $s$ in the planted clique, even though $E(V_s,V_s)$ is distributed as $\Binom(\binom{\ell_s}{2},q)$ which is not sufficiently close to the desired $\Binom(\binom{\ell_s}{2},p)$, after averaging over the random partition $\{V_s\}$, the total variation distance becomes negligible.
%The first step shows that  $P'_{\ell_s \ell_t}$ and $Q'_{\ell_s \ell_t}$ are well-defined probability measures with $(1-\gamma) Q'_{\ell_s \ell_t}+ \gamma P'_{\ell_s \ell_t} = Q_{\ell_s \ell_t}= \Binom(\ell_s \ell_t,q)$ and consequently $P_{\tG|H_0^C}=\Prob_0$. The second step shows that $d_{\rm TV} ( P'_{\ell_s \ell_t}, P_{\ell_s\ell_t}) \le  4 (8q \ell^2)^{(m_0+1)}.$
\begin{proposition}\label{prop:reduction}
Let $\ell, n \in \naturals$, $k \in [n]$ and $\gamma \in (0,\frac{1}{2} ]$. Let $N= \ell n$, $K=k\ell$, $p=2q$ and $m_0= \lfloor \log_2 (1/\gamma) \rfloor$. Assume that $16 q \ell^2 \le  1$ and $k \geq 6e \ell$.
If $G \sim \mathcal{G}(n, \gamma)$, then $\tG \sim \calG(N,q)$, \ie, $P_{\tG|H_0^C}=\Prob_0$.
If $G \sim \calG(n,k,1,\gamma)$, then
\begin{align}
d_{\rm TV} \left(P_{\tG|H_1^C}, \Prob_1 \right) \le e^{-\frac{K}{12}} + 1.5 k e^{-\frac{\ell}{18}} + 2 k^2 (8q\ell^2)^{m_0+1} + 0.5 \sqrt{e^{72e^2 q \ell^2} -1} + \sqrt{0.5k} e^{-\frac{\ell}{36}}. \label{eq:defxi}
\end{align}
%where $\widetilde{\Prob}_1$ is the distribution of $\tA$ under $$.
\end{proposition}

An immediate consequence of \prettyref{prop:reduction} is the following result (proved in \prettyref{sec:pf-test}) showing that any $\PDS$ solver induces a solver for a corresponding instance of the $\PC$ problem.

\begin{proposition}
Let the assumption of \prettyref{prop:reduction} hold.
  Suppose $\phi: \{0,1\}^{\binom{N}{2}} \to \{0, 1\}$ is a test for  $\PDS(N,K,2q,q)$  with Type-I+II error probability $\eta$.
%  , \ie,
%  \begin{align*}
%  \Prob_0 \{  \phi(A') =1 \} + \Prob'_1 \{ \phi(A')=0 \} \le \eta.
%  \end{align*}
%  where $A'$ is the adjancency matrix.
%  Fix a positive constant $\gamma \le 1/2$ and let $m_0=\lfloor \log_2 (1/\gamma) \rfloor$. Fix positive integers $k \le n$ such that $k/n=K/N$ and
%  $\ell=N/n \ge 1$.
  %Let $n=N/\ell$, $k=K/\ell$ \nb{need to floor or ceil??} and .
%  Suppose that $16 q \ell^2 \le 1$ and $k\ge6e \ell$.
  Then $G \mapsto \phi(\tG)$ is a test for the $\PC(n,k,\gamma)$ whose Type-I+II error probability is upper bounded by $\eta+ \xi$ with $\xi$ given by the right-hand side of \prettyref{eq:defxi}.
  %\begin{align*}
%  \xi=e^{-\frac{K}{12}} + 1.5 k e^{-\frac{\ell}{18}} + k^2 (8q\ell^2)^{m_0+1} + 0.5 \sqrt{e^{72e^2 q \ell^2} -1} + \sqrt{0.5k} e^{-\frac{\ell}{36}}.
%  \end{align*}
  \label{prop:test}
    \end{proposition}

The following theorem establishes the computational limit of the $\PDS$ problem as shown in \prettyref{fig:phase}.
\begin{theorem}
    Assume \prettyref{hyp:HypothesisPlantedClique} holds for a fixed $0<\gamma \le 1/2$.  Let $m_0=\lfloor \log_2 (1/\gamma) \rfloor$.
  %  Fix $0<\delta<1/2$.
%     and $ \alpha, \beta \in [0,1]$ such that
Let $ \alpha>0$ and $0<\beta<1$ be such that
  %  \begin{align}
%    \alpha<\beta< \sup_{\delta>0} \min \left \{  \frac{2+m_0\delta }{4+2\delta} \alpha, \; \frac{1}{2}-\delta + \frac{1+2\delta}{4+2\delta} \alpha \right\}.   \label{eq:HardRegimeExpression}
%    \end{align}
      \begin{align}
    \alpha<\beta<  \frac{1}{2} + \frac{m_0\alpha+4}{4m_0\alpha+4} \alpha - \frac{2}{m_0 \alpha}.   \label{eq:HardRegimeExpression}
    \end{align}
%    Let $\ell \in \naturals$ and $q_\ell=\ell^{-(2+\delta)}$.
    Then there exists a sequence $\{(N_\ell,K_\ell,q_\ell)\}_{\ell \in \naturals}$ satisfying
    \begin{align*}
    \lim_{\ell \to \infty} \frac{\log \frac{1}{q_\ell}  }{ \log N_\ell} =\alpha, \quad \lim_{\ell \to \infty}  \frac{\log K_\ell}{ \log N_\ell}= \beta
    \end{align*}
    such that for any sequence of randomized polynomial-time tests $\phi_{\ell}: \{0,1\}^{\binom{N_\ell}{2} } \to \{0,1\}$ for the $\PDS(N_\ell,K_\ell,2q_\ell,q_\ell)$ problem, the Type-I+II error probability is lower bounded by
    \begin{align*}
\liminf_{\ell \to \infty} \Prob_0\{ \phi_\ell(G')=1\}+ \Prob_1 \{ \phi_\ell(G')=0 \} \geq 1,
%\label{eq:main-risk}
\end{align*}
where $G' \sim \calG(N,q)$ under $H_0$ and $G' \sim \calG(N,K,p,q)$ under $H_1$.
%\nb{here write $A'$ is really weird...}
Consequently, if \prettyref{hyp:HypothesisPlantedClique} holds for all $0<\gamma \le 1/2$, then the above holds for all $ \alpha>0$ and $0<\beta<1$ such that
\begin{align}
    \alpha<\beta<\beta^{\sharp} (\alpha) \triangleq \frac{1}{2} + \frac{\alpha}{4}. \label{eq:computationallimit}
    \end{align}
\label{thm:main}
\end{theorem}
\begin{remark}
 Consider the asymptotic regime given by \prettyref{eq:scaling}. The function $\beta^\sharp$ in \prettyref{eq:computationallimit} gives the computational barrier for the $\PDS(N,K,p,q)$ problem (see \prettyref{fig:phase}). Compared to the statistical limit $\beta^*$ given in \prettyref{eq:detectionlimit}, we note that $\beta^*(\alpha)<\beta^\sharp(\alpha)$ if and only if $\alpha < \frac{2}{3}$, in which case computational efficiency incurs a significant penalty on the detection performance.
%In other words, in the \emph{moderately sparse} graph regime with $\alpha<\frac{2}{3}$, the computational intensive scan test is able to detect far smaller dense subgraphs than any computationally efficient test. On the other hand, in the \emph{very sparse} graph regime with $\alpha>2/3$, $\beta^{\sharp} $ and $\beta^\ast$ coincide, and the linear test statistic achieves the statistical limit in linear time.
Interestingly, this phenomenon is in line with the observation reported in \cite{MaWu13} for the noisy submatrix detection problem, where the statistical limit can be attained if and only if the submatrix size exceeds the $(2/3)^{\Th}$ power of the matrix size.
%  Assume the \PC Hypothesis holds for any positive $\gamma \in (0,1)$. Then $m_0$ in \prettyref{cor:compuationalowerbound} can be chosen to be arbitrarily large and thus we conclude that there is no efficiently computable solver if
%\begin{align}
%\alpha<\beta<\beta^{\sharp} (\alpha) =\frac{1}{2} + \frac{\alpha}{4}. \label{eq:computationallimit}
%\end{align}
%Therefore, $\beta^{\sharp} (\alpha)$ characterizes the computational limit.
\end{remark}

\section{Extensions and Open Problems}	\label{sec:extend}
In this section, we discuss the extension of our results to:
(1) the planted dense subgraph recovery and \DKS problem;
(2) the \PDS problem where the planted dense subgraph has a deterministic size.
(3) the bipartite \PDS problem;

\subsection{Recovering Planted Dense Subgraphs and \DKS Problem} \label{sec:recovery}
Closely related to the \PDS detection problem is the recovery problem, where given a graph generated from $\calG(N,K,p,q)$, the task is to recover the planted dense subgraph. As a consequence of our computational lower bound for detection, we discuss implications on the tractability of the recovery problem as well as the closely related $\DKS$ problem as illustrated in \prettyref{fig:planteddensesubgraph}.

Consider the asymptotic regime of \prettyref{eq:scaling}, where it has been shown \cite{ChenXu14,ames2013robust} that recovery is possible if and only if $\beta>\alpha$ and $\alpha <1$. Note that in this case the recovery problem is harder than finding the $\DKS$, because if the planted dense subgraph is recovered with high probability, we can obtain a $(1+\epsilon)$-approximation of the densest $K$-subgraph for any $\epsilon>0$ in polynomial time.\footnote{If the planted dense subgraph size is smaller than $K$, output any $K$-subgraph containing it; otherwise output any of its $K$-subgraph.} Results in \cite{ChenXu14,ames2013robust} imply that the planted dense subgraph can be recovered in polynomial time in the simple (green) regime of \prettyref{fig:planteddensesubgraph} where $ \beta> \frac{1}{2}+\frac{\alpha}{2}$. Consequently $(1+\epsilon)$-approximation of the \DKS can be found efficiently in this regime.

Conversely, given a polynomial time $\eta$-factor approximation algorithm to the \DKS problem with the output $\widehat{S}$, we can distinguish $H_0: G \sim \calG(N,q)$ versus $H_1: G \sim \calG(N,K,p=cq,q)$ if $\beta>\alpha$ and $c>\eta$ in polynomial time as follows: Fix any positive $\epsilon>0$ such that $(1-\epsilon)c>(1+\epsilon)\eta$. Declare $H_1$ if the density of $\widehat{S}$ is larger than $(1+\epsilon)q$ and $H_0$ otherwise. Assuming $\beta>\alpha$, one can show that the density of $\widehat{S}$ is at most $(1+\epsilon) q$ under $H_0$ and at least $(1-\epsilon)p/\eta$ under $H_1$. Hence, our computational lower bounds for the \PC problem imply that the densest $K$-subgraph as well as the planted dense subgraph is hard to approximate to any constant factor if $\alpha<\beta < \beta^{\sharp}(\alpha)$ (the red regime in \prettyref{fig:phase}). Whether $\DKS$ is hard to approximate with any constant factor in the blue regime of $ \beta^{\sharp}(\alpha) \vee \alpha \le \beta \le \frac{1}{2}+\frac{\alpha}{2}$ is left as an interesting open problem.

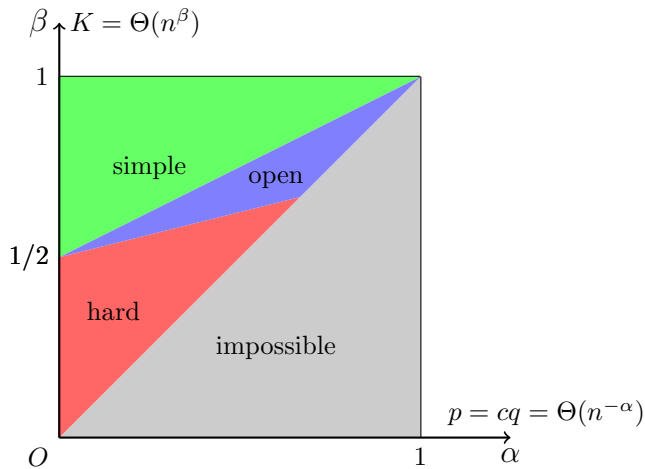
\begin{figure}[h!]
     \begin{center}
\scalebox{1}{
\begin{tikzpicture}[scale = 2.4, font = \small, thick]
\draw (2, 0) -- (2, 2);
\draw (0, 2) -- (2, 2);
\node [left] at (0, 2) {$1$};
\node [below] at (2, 0) {$1$};
\node [above] at (2.7,0) {$p=cq=\Theta(n^{-\alpha})$};
\node [right] at (0,2.3) {$K=\Theta(n^{\beta})$};
\node [left] at (0, 1) {$1/2$};
\draw[green] (0,1)--(0,2);
\draw[red] (0,0)--(0,1);
\path[fill = black!20] (0, 0) -- (2, 2) -- (2, 0) -- cycle;
\node at (1.2,0.5){impossible};
\path[fill= red!60] (0,0) --(4/3,4/3) -- (0,1)-- cycle;
\path[fill=green!60] (0,1)--(2,2) -- (0,2)-- cycle;
\path[fill= blue!50] (0,1) --(2,2) -- (4/3,4/3)-- cycle;
\node at (0.5,1.5){simple};
\node [left] at (0, 1) {$1/2$};
%\draw[thin, dashed] (0, 1) -- (4/3, 4/3);
%\node [right] at (2,1.5){$3/4$};
\node at (0.3,0.7) {hard};
\node at (1.2,1.43) {open};
\draw[->] (0, 0) node [below left] {$O$}-- (2.5, 0) node [below]{\large $\alpha$};
\draw[->] (0, 0) -- (0, 2.3) node [left]{\large $\beta$};
\end{tikzpicture}}
\end{center}
\caption{The simple (green), hard (red), impossible (gray) regimes for \textbf{recovering} planted dense subgraphs, and the hardness in the blue regime remains open.}
\label{fig:planteddensesubgraph}
\end{figure}

\subsection{\PDS Problem with a Deterministic Size} \label{sec:fixedsize}
In the \PDS problem with a deterministic size $K$, the null distribution corresponds to the Erd\H{o}s-R\'enyi graph $G(N,q)$; under the alternative, we choose $K$ vertices uniformly at random to plant a dense subgraph with edge probability $p$. Although the subgraph size under our \PDS model is binomially distributed, which, in the asymptotic regime \prettyref{eq:scaling}, is sharply concentrated near its mean $K$, it is not entirely clear whether these two models are equivalent.
%So far our reduction scheme in \prettyref{sec:computationallimits} does not extend to the fixed-size model; the main technical hurdle is in controlling the total variation distance under the alternative due to high dependency introduced by the diagonal blocks of the adjacency matrix $\tA$.
Although our reduction scheme in \prettyref{sec:computationallimits} extends to the fixed-size model with $V_1^n$ being the random $\ell$-partition of $[N]$ with $|V_t|=\ell$ for all $t \in [n]$, so far we have not been able to prove the alternative distributions are approximately matched:
 The main technical hurdle lies in controlling the total variation between the distribution of $\{E(V_t,V_t), t \in [n] \}$ after averaging over the
 random $\ell$-partition $\{V_t\}$ and the desired distribution.

%Nonetheless, if one restrict oneself to tests that are monotonically increasing (with respect to subgraph ordering), \ie, $\phi(G) \geq \phi(G')$ if $G'$ can be obtained from $G$ by adding edges.
% We say a test $\phi: \{0,1\}^{N \times N} \to \{0,1\}$ is {\it monotone}, if $\phi (A) \ge \phi(A')$ when $A_{ij}' \ge A_{ij}$ for all $i,j \in [N]$.

Nonetheless, our result on the hardness of solving the PDS problem extends to the case of
deterministic dense subgraph size if the tests are required to be monotone. (A test
$\phi$ is monotone if $\phi(G)=1$ implies $\phi(G')=1$ whenever $G'$ is obtained by adding edges to $G$.)
It is intuitive to assume that any reasonable test should be more likely to declare the existence of the planted dense subgraph if the graph contains more edges, such as the linear and scan test defined in \prettyref{eq:tests}.
%accept the alternative hypothesis with more edges
%and thus be monotone.
Moreover, by the monotonicity of the likelihood ratio, the statistically optimal test is also monotone.
If we restrict our scope to monotone tests, then
 our computational lower bound implies that for the \PDS problem with a deterministic size, there is no efficiently computable monotone test in the hard regime of $\alpha<\beta<\beta^{\sharp}$ in \prettyref{fig:phase}. In fact, for a given monotone polynomial-time solver $\phi$ for the \PDS problem with size $K$, the $\PDS(N,2K,p,q)$ can be solved by $\phi$ in polynomial time because with high probability the planted dense subgraph is of size at least $K$.
  It is an interesting open problem to prove the computational lower bounds without restricting to monotone tests, or prove the optimal polynomial-time tests are monotone. We conjecture that the computational limit of $\PDS$ of fixed size is identical to that of the random size, which can indeed by established in the bipartite case as discussed in the next subsection.

Finally, we can show that  the \PDS \emph{recovery} problem with a deterministic planted dense subgraph size $K$ is
 computationally intractable if $\alpha<\beta < \beta^{\sharp}(\alpha)$ (the red regime in \prettyref{fig:phase}). This follows from the fact that given a polynomial-time algorithm for the \PDS recovery problem with size $K$,
 we can construct a polynomial-time solver for $\PDS(N,K, p=cq, q)$ if $\alpha<\beta$ (See Appendix \ref{sec:PDSRecovery} for a formal statement and the proof).

\subsection{Bipartite \PDS Problem} \label{sec:bipartite}
Let $\mathcal{G}_{\rm b}(N,q)$ denote the bipartite Erd\H{o}s-R\'enyi random graph model with $N$ top vertices and $N$ bottom vertices.
Let $\mathcal{G}_{\rm b}(N,K,p,q)$ denote the bipartite variant of the planted densest subgraph model in \prettyref{def:HypTesting} with a planted dense subgraph of $K$ top vertices and $K$ bottom vertices on average. The bipartite \PDS problem with parameters $(N,K,p,q)$, denoted by $\BPDS(N,K,p,q)$, refers to the problem of testing $H_0: G \sim \mathcal{G}_{\rm b}(N,q)$ versus $H_1: G \sim \mathcal{G}_{\rm b} (N,K,p,q)$.

%\nb{Jiaming, this $\calG(N,N,K,K,p,q)$ and etc look like typos... Maybe it's better to use another notation like $\calG_{\rm b}$ to denote the bipartite ensemble?}

 Consider the asymptotic regime of \prettyref{eq:scaling}. Following the arguments in \prettyref{sec:statisticallimits}, one can show that the statistical limit is given by $\beta^\ast$ defined in \prettyref{eq:detectionlimit}.
To derive computational lower bounds, we use the reduction from the bipartite \PC problem with parameters $(n,k,\gamma)$, denoted by $\BPC(n,k,\gamma)$, which tests $H_0: G \sim \mathcal{G}_{\rm b}(n,\gamma)$ versus $H_1: G \sim \mathcal{G}_{\rm b} (n,k,\gamma)$, where $\mathcal{G}_{\rm b} (n,k,\gamma)$ is the bipartite variant of the planted clique model with a planted  bi-clique of size $k \times k$. The \BPC Hypothesis refers to the assumption that for some constant $0<\gamma \le 1/2$, no sequence of randomized polynomial-time tests for \BPC succeeds if $\limsup_{n \to \infty} \frac{\log k_n }{ \log n} < 1/2$. The reduction scheme from $\BPC(n,k,\gamma)$ to $\BPDS(N,K,2q,q)$  is analogue to the scheme used in non-bipartite case.
%\begin{enumerate}
%\item (Graph Inflation): Given a bipartite graph $G$ with $n$ top vertices indexed by $[n]$, $n$ bottom vertices indexed by $[n]$ and the adjacency matrix given by $A$, we define an enlarged bipartite graph $\widetilde{G}$ with $N$ top vertices indexed by $[N]$, $N$ bottom vertices indexed by $[N]$ as follows: Let $U_1^n$ (resp.\ $V_1^n$) denote the sets of vertices in bin $t \in [n]$ by throwing $N$ top (resp.\ bottom) vertices into $n$ bins. For all $(s,t) \in [n]^2$, we generate the number of edges between $U_s$ and $V_t$, denoted by
% $E(U_s,V_t)$, independently as \prettyref{eq:edgedist}.
%% \begin{align*}
%%    E(U_s, V_t) \sim \left\{
%%    \begin{array}{rl}
%%    P'_{\ell_s \ell_t} & \text{if } A_{st}=1, \ell_s, \ell_t \le 2\ell\\
%%    Q'_{\ell_s \ell_t} & \text{if } A_{st}=0, \ell_s, \ell_t \le 2 \ell \\
%%    Q_{\ell_s \ell_t} & \text{if } \max\{ \ell_s, \ell_t \}> 2 \ell.
%%    \end{array} \right.
%%    \end{align*}
% Then independently of everything else, we connect top vertices in $U_s$ and bottom vertices in $V_t$ uniformly at random such that the total number of edges is $E(U_s,V_t)$.
%\end{enumerate}
The proof of computational lower bounds in bipartite graph is much simpler. In particular, under the null hypothesis, $G \sim \mathcal{G}_{\rm b}(n, \gamma)$ and one can verify that $\widetilde{G} \sim \mathcal{G}_{\rm b}(N,q)$. Under the alternative hypothesis, $G \sim \mathcal{G}_{\rm b}(n,k,\gamma)$. \prettyref{lmm:TotalVariationBound} directly implies that the total variation distance between the distribution of $\tG$ and $\mathcal{G}_{\rm b} (N,K,2q,q)$ is on the order of $k^2 (q \ell^2)^{(m_0+1)}$. Then, following the arguments in \prettyref{prop:test} and \prettyref{thm:main}, we conclude that if the \BPC Hypothesis holds for any positive $\gamma$, then no efficiently computable test can solve $\BPDS(N,K,2q,q)$ in the regime $\alpha<\beta< \beta^{\sharp} (\alpha)$ given by \prettyref{eq:computationallimit}. The same conclusion also carries over to the bipartite \PDS problem with a deterministic size $K$ and the statistical and computational limits shown in \prettyref{fig:phase} apply verbatim.
%, our approach and results can be readily extended, and it can be shown that there is no efficiently computable test in the regime $\alpha<\beta<\beta^{\sharp}$.

\bibliographystyle{abbrv}
\bibliography{strings,refs,ballsandbins,planted}

\appendix

\section{Proofs} \label{sec:pf}
\subsection{Proof of \prettyref{prop:lowerbound}}
\label{sec:pf-lb}
\begin{proof}
Let $\Prob_{A | |S|}$ denote the distribution of $A$ conditional on $|S|$ under the alternative hypothesis. Since $|S| \sim \Binom(N,K/N)$, by the
Chernoff bound, $\Prob[|S| >2K] \le \exp(-K/8) $. Therefore,
\begin{align}
d_{\rm TV} (\Prob_0, \Prob_1 ) &= d_{\rm TV} (\Prob_0, \mathbb{E}_{|S|} [\Prob_{A | |S|}]  )  \nonumber \\
& \le \mathbb{E}_{|S|} \left[ d_{\rm TV} (\Prob_0, \Prob_{A | |S|} )  \right] \nonumber \\
& \le \exp(-K/8) + \sum_{ K'\le 2K} d_{\rm TV} (\Prob_0, \Prob_{A | |S|=K'} ) \Prob [|S|=K'], \label{eq:dtv1}
\end{align}
where the first inequality follows from the convexity of $(P,Q) \mapsto d_{\rm TV} (P,Q)$,
%Consider a channel which takes a graph $x$ as input and generates the output $y$ by deleting each edge independently at random with probability $1-p$. Suppose the input $X \sim G(N,K,1,1/2)$, then the output $Y \sim G(N,K,p,q)$.
Next we condition on $|S|=K'$ for a fixed $K' \le 2K$. Then $S$ is uniformly distributed over all subsets of size $K'$. Let $\tS $ be an independent copy of $S$. % conditioned on $|S|=K'$.
Then $|S \cap \tS| \sim \mathrm{Hypergeometric}(N,K',K')$. By the definition of the $\chi^2$-divergence and Fubini's theorem,
\begin{align*}
\chi^2(\Prob_{A | |S|=K'} \| \Prob_0 )
 &= \int \frac{ \Expect_S[P_{A | S}] \Expect_{\tS}[P_{A | \widetilde{S}}] }{\Prob_0}  -1 \\
 &= \mathbb{E}_{ S \Perp \widetilde{S} } \left[ \int \frac{ P_{A | S} P_{A | \widetilde{S}} }{\Prob_0} \right] -1 \\
&= \mathbb{E}_{ S \Perp \widetilde{S} } \left[  \left( 1+ \frac{(p-q)^2}{q(1-q)} \right)^{\binom{|S \cap \widetilde{S} |}{2} } \right] -1 \\
& \le \mathbb{E}_{ S \Perp \widetilde{S} } \left[  \exp \left ( \frac{(c-1)^2 q}{1-q}\binom{|S \cap \widetilde{S} |}{2} \right) \right] -1 \\
& \overset{(a)}{\le} \mathbb{E} \left[  \exp \left( (c-1)c q |S \cap \widetilde{S} |^2 \right) \right] -1 \\
& \overset{(b)} {\le} \tau (Cc^2) -1,
\end{align*}
where $(a)$ is due to the fact that $q = \frac{p}{c} \leq \frac{1}{c}$; $(b)$ follows from \prettyref{lmm:H} in \prettyref{app:H}
with an appropriate choice of function $\tau:\reals_+ \to \reals_+$ satisfying $\tau(0+)=1$. Therefore, we get that
\begin{align}
2 d^2_{\rm TV} (\Prob_0, \Prob_{A | |S|=K'} ) \le  \log ( \chi^2( \Prob_{A | |S|=K'} \| \Prob_0 )  +1 ) \le
\log ( \tau(C c^2) ) , \label{eq:dtv2}
\end{align}
Combining \prettyref{eq:dtv1} and \prettyref{eq:dtv2} yields \prettyref{eq:lb-tv} with $h \triangleq \log \circ \tau$.
%and thus $d_{\rm TV} (\Prob_0, \Prob_1 ) \le \tau (Cc^2)+ \exp(-K/8)$ for some function $\tau:\reals_+ \to \reals_+$ satisfying $\tau(0+)=0$.
\end{proof}

\subsection{Proof of \prettyref{prop:upperbound}}
\label{sec:pf-ub}
\begin{proof}
%\nb{jiaming, either add $i<j$ into the defintion of Tlin and Tscan, or modify the proof. Otherwise, the distribution is twice binomial.}
Let $C>0$ denote a constant whose value only depends on $c$ and may change line by line.
Under $\mathbb{P}_0$, $T_{\rm lin} \sim \Binom \left(\binom{N}{2}, q \right)$. By the Bernstein inequality,
\begin{align*}
\mathbb{P}_0 [ T_{\rm lin} > \tau_1] \le \exp \left( - \frac{ \binom{K}{2}^2 (p-q)^2 /4 }{ 2 \binom{N}{2}q + \binom{K}{2} (p-q)/3 } \right) \le \exp \left( - C \frac{K^4q}{N^2}\right).
\end{align*}
Under $\mathbb{P}_1$,  Since $|S| \sim \Binom(N,K/N)$, by the Chernoff bound, $\Prob_1[|S| <0.9K] \le \exp(-K/200)$.
Conditional on $|S|=K'$ for some $K' \ge 0.9K$, then $T_{\rm lin}$ is distributed as an independent sum of $\Binom \left(\binom{K'}{2}, p \right)$ and $\Binom\left(\binom{N}{2}-\binom{K'}{2}, q\right)$. By the multiplicative Chernoff bound (see, \eg, \cite[Theorem 4.5]{Mitzenmacher05}),
\begin{align*}
\mathbb{P}_1[ T_{\rm lin} \le \tau_1] & \le \Prob_1[|S| <0.9K] + \exp \left( - \frac{  \left(2 \binom{K'}{2} - \binom{K}{2} \right)^2 (p-q)^2 }{ 8 \left( \binom{N}{2}q + \binom{K'}{2} (p-q) \right)  } \right) \\
& \le \exp \left(-\frac{K}{200} \right) + \exp \left( -C \frac{K^4q}{N^2}\right).
\end{align*}
For the scan test statistic, under the null hypothesis, for any fixed subset $S$ of size $K$, $\sum_{i,j \in S} A_{ij} \sim \Binom\left(\binom{K}{2}, q\right) $.
By the union bound and the Bernstein inequality,
\begin{align*}
\mathbb{P}_0 [ T_{\rm scan} > \tau_2] \le \binom{N}{K} \mathbb{P}_0 [ \sum_{1 \le i<j \le K} A_{ij} > \tau_2 ] & \le \left(\frac{Ne}{K}\right)^K \exp \left(  - \frac{\binom{K}{2}^2 (p-q)^2 /4}{ 2 \binom{K}{2}q + \binom{K}{2} (p-q)/3 } \right) \\
&\le  \exp \left( K \log \frac{Ne}{K} - C K^2 q \right).
\end{align*}
Under the alternative hypothesis, conditional on $|S|=K'$ for some $K' \ge 0.9K$, $ \sum_{i,j \in S} A_{ij} \sim \Binom\left(\binom{K'}{2}, p\right) $ and thus $T_{\rm scan}$ is stochastically dominated by $\Binom\left(\binom{K' \wedge K } {2}, p\right)$. By the multiplicative Chernoff bound,
\begin{align*}
\mathbb{P}_1 [ T_{\rm scan} \le \tau_2] &\le
\Prob_1[|S| <0.9K] + \exp \left(- \frac{ \left( 2 \binom{K' \wedge K}{2} - \binom{K}{2} \right)^2 (p-q)^2}{8 \binom{K'\wedge K}{2}p } \right)  \\
&\le  \exp \left(-\frac{K}{200} \right) + \exp \left( - C K^2 q \right).
\end{align*}
\end{proof}

\subsection{Proof of \prettyref{prop:reduction}}
\label{sec:pf-reduction}
We first introduce several key auxiliary results used in the proof. The following lemma ensures that $P'_{\ell_s \ell_t}$ and $Q'_{\ell_s \ell_t}$ are well-defined under suitable conditions and that $P'_{\ell_s \ell_t}$ and $P_{\ell_s,\ell_t}$ are close in total variation.
\begin{lemma}\label{lmm:TotalVariationBound}
Suppose that $p=2q$ and $16 q \ell^2 \le 1$. Fix $\{\ell_t\}$ such that $\ell_t \le 2 \ell$ for all $t \in [k]$.
Then for all $1 \le s <t \le k$,
$P'_{\ell_s \ell_t}$ and $Q'_{\ell_s \ell_t}$ are probability measures and
\begin{align*}
d_{\rm TV} ( P'_{\ell_s \ell_t}, P_{\ell_s\ell_t}) \le  4 (8q \ell^2)^{(m_0+1)}.
\end{align*}
\end{lemma}
\begin{proof}
Fix an $(s,t)$ such that $1 \le s<t \le k$. We first show that $P'_{\ell_s \ell_t}$ and $Q'_{\ell_s \ell_t}$ are well-defined. By definition, $\sum_{m=0}^{\ell_s\ell_t} P'_{\ell_s \ell_t} (m) = \sum_{m=0}^{\ell_s\ell_t} Q'_{\ell_s \ell_t} (m) = 1$ and
it suffices to show positivity, \ie,
\begin{align}
&P_{\ell_s\ell_t} (0) + a_{\ell_s \ell_t}  \ge 0, \label{eq:condition1}\\
&Q_{\ell_s\ell_t}(m)  \ge  \gamma P'_{\ell_s \ell_t}(m) , \quad \forall 0 \le m \le m_0. \label{eq:condition2}
\end{align}
Recall that $P_{\ell_s\ell_t} \sim \Binom(\ell_s \ell_t, p)$ and $Q_{\ell_s\ell_t} \sim \Binom(\ell_s \ell_t, q)$. Therefore,
\begin{align*}
Q_{\ell_s\ell_t} (m) = \binom{\ell_s \ell_t }{m} q^m (1-q)^{\ell_s \ell_t-m}, \quad
P_{\ell_s\ell_t} (m) = \binom{\ell_s \ell_t }{m} p^m (1-p)^{\ell_s \ell_t -m}, \; \forall 0 \le m \le \ell_s \ell_t,
\end{align*}
It follows that
\begin{align*}
\frac{1}{\gamma} Q_{\ell_s\ell_t} (m)- P_{\ell_s\ell_t} (m)= \frac{1}{\gamma } \binom{\ell_s \ell_t} {m} q^m (1-2q)^{\ell_s \ell_t -m } \left[ \left( \frac{1-q}{1-2q}\right)^{\ell_s \ell_t -m} - 2^m \gamma  \right].
\end{align*}
Recall that $m_0=\lfloor \log_2 (1/\gamma) \rfloor$ and thus $ Q_{\ell_s\ell_t}(m) \ge  \gamma P_{\ell_s\ell_t} (m)$ for all $m \le m_0$.
Furthermore,
\begin{align*}
Q_{\ell_s\ell_t} (0) =  (1-q)^{\ell_s \ell_t } \ge (1-q \ell_s \ell_t ) \ge 1-4q\ell^2  \ge \frac{3}{4} \ge \gamma \geq \gamma P'_{\ell_s\ell_t} (0) ,
\end{align*}
and thus \prettyref{eq:condition2} holds.
Recall that
\begin{align*}
a_{\ell_s \ell_t}=\sum_{m_0<m \le  \ell_s\ell_t}  \left(P_{\ell_s \ell_t}(m) - \frac{1}{\gamma} Q_{\ell_s \ell_t}(m)  \right)
\end{align*}
Since $2^{m_0+1} \gamma >1$ and $8 q \ell^2 \leq 1/2$, it follows that
\begin{align}
\frac{1}{\gamma} \sum_{m_0<m \le  \ell_s\ell_t}  Q_{\ell_s \ell_t}(m)  \le   \frac{1}{\gamma}  \sum_{m_0<m \le  \ell_s\ell_t} \binom{\ell_s \ell_t }{m} q^m \le   \sum_{m>m_0 } ( 2 \ell_s \ell_t q)^m  \le   2 (8q\ell^2 )^{(m_0+1) }, \label{eq:bound1}
\end{align}
and therefore $a_{\ell_s \ell_t}  \ge -1/2$. Furthermore,
\begin{align*}
P_{\ell_s\ell_t} (0) =  (1-p)^{\ell_s \ell_t } \ge 1-p \ell_s \ell_t  \ge 1-8q\ell^2 \ge 1/2,
\end{align*}
and thus \prettyref{eq:condition1} holds.

Next we bound $d_{\rm TV} \left( P'_{\ell_s \ell_t}, P_{\ell_s\ell_t} \right)$. Notice that
\begin{align}
\sum_{m_0<m \le  \ell_s\ell_t}  P_{\ell_s \ell_t}(m) \le \sum_{m_0<m \le  \ell_s\ell_t}  \binom{\ell_s \ell_t }{m} p^m \le \sum_{m >m_0} (\ell_s \ell_t p)^m \le  2 (8 q\ell^2)^{(m_0+1) }. \label{eq:bound2}
\end{align}
Therefore, by the definition of the total variation distance and $a_{\ell_s\ell_t}$,
\begin{align*}
d_{\rm TV} ( P'_{\ell_s \ell_t}, P_{\ell_s\ell_t} )
= & ~ \frac{1}{2} |a_{\ell_s\ell_t}| + \frac{1}{2} \sum_{m_0<m \le  \ell_s\ell_t} \left|P_{\ell_s \ell_t}(m) - \frac{1}{\gamma}  Q_{\ell_s \ell_t}(m)\right|	 \nonumber \\
\leq & ~ 	\sum_{m_0<m \le  \ell_s\ell_t} \left(P_{\ell_s \ell_t}(m) + \frac{1}{\gamma}  Q_{\ell_s \ell_t}(m)\right) \le 4 (8 q\ell^2)^{(m_0+1) },
% \le   \max \left\{\sum_{m_0<m \le  \ell_s\ell_t}P_{\ell_s \ell_t}(m) , \sum_{m_0<m \le  \ell_s\ell_t} \frac{1}{\gamma}  Q_{\ell_s \ell_t}(m)  \right\} \le 2 (8 q\ell^2)^{(m_0+1) },
\end{align*}
where the last inequality follows from \prettyref{eq:bound1} and \prettyref{eq:bound2}.
\end{proof}

The following lemma is useful for upper bounding the total variation distance between a truncated mixture of product distribution $P_Y$ and a product distribution $Q_Y$.
\begin{lemma}
Let $P_{Y|X}$ be a Markov kernel from $\calX$ to $\calY$ and denote the marginal of $Y$ by $P_Y = \Expect_{X\sim P_X}[P_{Y|X}]$. Let $Q_Y$ be such that $P_{Y|X=x} \ll Q_Y$ for all $x$. Let $E$ be a measurable subset of $\calX$.
Define $g:\calX^2 \to \bar{\reals}_+$ by
\[
g(x,\tx) \triangleq \int \frac{\diff P_{Y|X=x} \diff P_{Y|X=\tx}}{\diff Q}.
\]
Then
\begin{equation}
%\prob{X\not \in E}
\dTV(P_Y,Q_Y) \leq \frac{1}{2} P_X(\comp{E}) + \frac{1}{2} \sqrt{\expect{g(X,\widetilde{X}) \Indc_E(X) \Indc_E(\tX)} - 1 + 2 P_X(\comp{E})},
%\dTV(P_Y,Q_Y) \leq P_X(\comp{E}) + \frac{1}{2} \sqrt{\expect{g(X,\tilde{X}) \Indc_E(X) \Indc_E(\tX)} - 1},
	\label{eq:chi2tr}
\end{equation}
where $\tX$ is an independent copy of $X\sim P_X$.
	\label{lmm:chi2tr}
\end{lemma}
\begin{proof}
By definition of the total variation distance,
\[
\dTV(P_Y,Q_Y) = \frac{1}{2} \|P_Y-Q_Y\|_1 \leq \frac{1}{2} \|\Expect[P_{Y|X}]- \Expect[P_{Y|X} \indc{X\in E}]\|_1+\frac{1}{2} \|\Expect[P_{Y|X} \indc{X\in E}]-Q_Y\|_1,
\]
where the first term is $\|\Expect[P_{Y|X}]- \Expect[P_{Y|X} \indc{X\in E}]\|_1=\|\Expect[P_{Y|X} \indc{X\not\in E}]\|_1=\prob{X\not\in E}$. The second term is controlled by
\begin{align}
%	& ~ 	\|\Expect[P_{Y|X} \indc{X\in E}]-Q_Y\|_1^2\nonumber \\
	\|\Expect[P_{Y|X} \indc{X\in E}]-Q_Y\|_1^2
= & ~ 	\left( \expects{\left|\frac{\expect{P_{Y|X}\indc{X\in E}}}{Q_Y}  - 1 \right|}{Q_Y} \right)^2\nonumber \\
\leq & ~ 	\expects{\pth{\frac{\expect{P_{Y|X}\indc{X\in E}}}{Q_Y}  - 1 }^2}{Q_Y} \label{eq:t1}\\
= & ~ 	\expects{\pth{\frac{\expect{P_{Y|X}\indc{X\in E}}}{Q_Y}}^2}{Q_Y} + 1 - 2 \, \Expect[\expect{P_{Y|X}\indc{X\in E}}]\\
= & ~ 	\expect{g(X,\tilde{X}) \Indc_E(X) \Indc_E(\tX)} + 1 - 2 \, \prob{X\in E} \label{eq:t2},
\end{align}
where \prettyref{eq:t1} is Cauchy-Schwartz inequality, \prettyref{eq:t2} follows from Fubini theorem. This proves the desired \prettyref{eq:chi2tr}.
\end{proof}

Note that  $\{V_t: t\in [n]\}$ can be equivalently generated as follows: Throw balls indexed by $[N]$ into bins indexed by $[n]$ independently and uniformly at random; let $V_t$ denote the set of balls in the $t^{\Th}$ bin. Furthermore, Fix a subset $C \subset [n]$ and let $S=\cup_{t \in C} V_t$. Conditioned on $S$,
$\{V_t: t \in C \}$ can be generated by throwing balls indexed by $S$ into bins indexed by $C$ independently and uniformly at random.
We need the following negative association property \cite[Definition 1]{Desh98}.
% random variables
% which allows us to deal with the dependence.
%Let $\{\tV_t: t\in [n]\}$ denote an independent copy.
%The following lemma proves the negative association property \cite{Desh98} of the full vector $ \{ |V_s \cap \widetilde{V}_t |: s,t \in[k]\} $.
\begin{lemma}
Fix a subset $C \subset [n]$ and let $S=\cup_{t \in C} V_t$. Let $\{\tV_t: t\in C \}$ be an independent copy of $\{V_t: t \in C \}$ conditioned on $S$.
Then conditioned on $S$, the full vector $ \{ |V_s \cap \widetilde{V}_t |: s,t \in C  \} $ is negatively associated, i.e., for every two disjoint index sets $I, J \subset C \times C$,
\begin{align*}
\Expect[f(V_s \cap \widetilde{V}_t, (s,t) \in I ) g( V_s \cap \widetilde{V}_t, (s,t) \in J )  ] \le \Expect[f(V_s \cap \widetilde{V}_t, (s,t) \in I )] \Expect[g( V_s \cap \widetilde{V}_t, (s,t) \in J )],
\end{align*}
for all functions $f: \reals^{|I|} \to \reals$ and $g: \reals^{|J|} \to \reals$ that are either both non-decreasing or both non-increasing in every argument.
\label{lmm:NA}
\end{lemma}	
\begin{proof}
Define the indicator random variables $Z_{m,s,t}$ for $m \in S, s,t \in C$ as
\begin{align*}
Z_{m,s,t} = \left\{
    \begin{array}{rl}
    1 & \text{if the $m^\Th$ ball is contained in $V_s$ and $\widetilde{V}_t$},\\
    0& \text{otherwise }.
    \end{array} \right.
\end{align*}
By \cite[Proposition 12]{Desh98}, the full vector $\{ Z_{m,s,t}: m \in S, s,t \in C \}$ is negatively associated. By definition, we have
\begin{align*}
|V_s \cap \widetilde{V}_t |= \sum_{m \in S} Z_{m, s,t},
\end{align*}
which is a non-decreasing function of $\{Z_{m,s,t}: m \in S \}$. Moreover, for distinct pairs $(s,t) \neq (s',t')$,
the sets $\{(m,s,t): m \in S \}$ and $\{(m,s',t'): m \in S \}$ are disjoint. Applying \cite[Proposition 8]{Desh98} yields the desired statement.
\end{proof}
The negative association property of $\{ |V_s \cap \widetilde{V}_t |: s,t \in C  \} $ allows us to bound the expectation of any non-decreasing function of $\{ |V_s \cap \widetilde{V}_t |: s,t \in C  \} $ conditional on $C$ and $S$ as if they were independent \cite[Lemma 2]{Desh98}, i.e., for any collection of non-decreasing functions $\{f_{s,t}: s,t \in [n]\}$,
\begin{align}
\Expect \left[ \prod_{s,t \in C} f_{s,t} (|V_s \cap  \widetilde{V}_t| ) \biggm \vert  C, S\right] \le  \prod_{s,t \in C} \Expect \left[ f_{s,t} (|V_s \cap  \widetilde{V}_t| ) \biggm \vert  C, S\right]. \label{eq:NAbound}
\end{align}

\begin{lemma} \label{LemmaComparison}
Suppose that $X \sim \Binom(1.5K, \frac{1}{k^2})$ and $Y \sim \Binom( 3 \ell, \frac{e}{k})$ with $K=k\ell$ and $ k \ge 6e \ell$. Then for all $ 1 \le m \le 2 \ell-1$,
\begin{align*}
\mathbb{P} [ X=m ] \le \mathbb{P}[Y =m ],
\end{align*}
and $ \mathbb{P}[ X \ge 2\ell] \le   \mathbb{P}[Y =2 \ell]$.
\end{lemma}
\begin{proof}
In view of the fact that $(\frac{n}{m})^m \leq \binom{n}{m} \leq (\frac{en}{m})^m$, we have for $1 \le m \le 2 \ell$,
\begin{align*}
\mathbb{P}[ X =m ] &= \binom{1.5K}{m} \left( \frac{1}{k^2} \right)^m \left( 1- \frac{1}{k^2} \right)^{1.5K-m} \le \left( \frac{1.5eK}{mk^2} \right)^m.
\end{align*}
Therefore,
\begin{align*}
\mathbb{P}[X \ge 2 \ell] \le \sum_{m =2\ell}^\infty \left( \frac{1.5e\ell}{km} \right)^m \le  \sum_{m =2\ell}^\infty \left( \frac{3e}{4k} \right)^m \le \frac{(0.75e/k)^{2\ell}}{1-0.75e/k}.
\end{align*}
On the other hand, for $ 1 \le m \le 2 \ell-1$
\begin{align*}
\mathbb{P} [Y=m] & = \binom{3 \ell }{m} \left( \frac{e}{k} \right)^m \left(1 - \frac{e}{k} \right)^{3 \ell -m } \\
& \ge \left( \frac{3 e \ell }{m k} \right)^m \left(1 - \frac{3 e \ell }{k} \right) \\
& \ge  2^{m-1} \left( \frac{ 1.5 e \ell }{m k} \right)^m \ge \mathbb{P}[ X =m ].
\end{align*}
Moreover, $\mathbb{P}[Y =2 \ell] \ge \mathbb{P}[ X \ge 2\ell].$
\end{proof}

\begin{lemma}
	Let $T \sim \Binom(\ell,\tau)$ and $\lambda > 0$. Assume that $\lambda \ell \leq \frac{1}{16}$. Then
\begin{equation}
\Expect[\exp(\lambda T(T-1))] \leq	\exp\pth{16\lambda \ell^2 \tau^2 }.
%\exp\pth{\frac{16\lambda \ell^2}{k^2} }.
	\label{eq:decoup}
\end{equation}
	\label{lmm:decoup}
\end{lemma}
\begin{proof}
 Let $(\ntok{s_1}{s_\ell},\ntok{t_1}{t_\ell}) \, \iiddistr \, \Bern(\tau)$, $S = \sum_{i=1}^\ell  s_i$ and $T = \sum_{i=1}^\ell  t_i$.
 Next we use a decoupling argument to replace $T^2-T$ by $ST$:
\begin{align}
\expect{\exp\pth{\lambda T(T-1)}}
= & ~ \Expect\bqth{\exp\bpth{\lambda \sum_{i \neq j} t_i t_j}}	\nonumber \\
\leq & ~ \Expect\bqth{\exp\bpth{4 \lambda \sum_{i \neq j} s_i t_j}}, \label{eq:decouple}\\
\leq & ~ \expect{\exp\pth{4 \lambda S T}}, \nonumber
\end{align}
where
%	\prettyref{eq:CS} is by Cauchy-Schwartz inequality and
	\prettyref{eq:decouple} is a standard decoupling inequality (see, \eg, \cite[Theorem 1]{Vershynin11}). Since $\lambda T \leq \lambda \ell \leq \frac{1}{16}$ and $\exp(x)-1 \leq \exp(a) x$ for all $x \in [0,a]$, the desired \prettyref{eq:decoup} follows from
%		$\expect{\exp\pth{2 \lambda S}} = (1+\frac{1}{k} (\exp(2\lambda)-1))^\ell \leq \exp(\frac{2\lambda \ell}{k})$. Moreover,
%	\begin{align*}
%	\expect{\exp\pth{4 \lambda S T}}
%= & ~ \expect{\pth{1+\frac{1}{k} (\exp(4\lambda T)-1)}^\ell}	\nonumber \\
%\leq & ~ \expect{\pth{1+\frac{8\lambda T}{k} }^\ell}	\nonumber \\
%\leq & ~ \expect{\exp\pth{\frac{8\lambda \ell T}{k} }}	\nonumber \\
%= & ~ \pth{1+\frac{1}{k} \pth{\exp\pth{\frac{8\lambda \ell}{k} }-1}}^\ell	\nonumber \\
%\leq & ~ \exp\pth{\frac{16\lambda \ell^2}{k^2} }. \qedhere
%\end{align*}	
\begin{align*}
	\expect{\exp\pth{4 \lambda S T}}
= & ~ \expect{\pth{1+ \tau (\exp(4\lambda T)-1)}^\ell}	\nonumber \\
\leq & ~ \expect{\pth{1+ 8\tau\lambda T }^\ell}	\nonumber \\
\leq & ~ \expect{\exp\pth{8\tau\lambda \ell T }}	\nonumber \\
= & ~ \pth{1+ \tau\pth{\exp\pth{8\tau\lambda \ell }-1}}^\ell	\nonumber \\
\leq & ~ \exp\pth{16\tau^2\lambda \ell^2  }.
%\qedhere
\end{align*}	
	\end{proof}

\begin{proof}[Proof of \prettyref{prop:reduction}]
Let $[i,j]$ denote the unordered pair of $i$ and $j$.
For any set $I \subset [N]$, let $\calE(I)$ denote the set of unordered pairs of distinct elements in $I$, i.e., $\calE(I)=\{[i,j]: i,j \in S, i\neq j\}$, and let $\calE(I)^c=\calE( [N] ) \setminus \calE(I)$.
%For $U,V\subset [N]$, let $\tP_{UV}$ denote the edge distribution of the subgraph $\tG_{UV}$ of $\tG$ induced by vertices in $U$ and $V$.  %$\tA_{UV}=(\tA_{ij})_{i\in U,j\in V,i<j}$.
For $s, t \in [n]$ with $s \neq t$, let $\tG_{V_sV_t}$ denote the bipartite graph where the set of left (right) vertices is $V_s$ (resp.\xspace $V_t$) and the set of edges is the set of edges in $\tG$ from vertices in $V_s$ to vertices in $V_t$. For $s \in [n]$, let $\tG_{V_sV_s}$ denote the subgraph of $\tG$ induced by $V_s$. Let $\tP_{V_sV_t}$ denote the edge distribution of $\tG_{V_sV_t}$ for $s, t \in [n]$.
%Let $\tP_{UU}$ denote the distribution of (the upper triangular part of) the submatrix $\tA_{UU}=(\tA_{ij})_{[i,j]\in \calE(U)}$.
%The proof is divided into three parts. Step 1 shows that under the null hypothesis, $\tP_{V_sV_t}$ is the product distribution formed by $\Bern(q)$ for all $1\le s\le t \le n$ and therefore $P_{\tG|H_0^C}=\Prob_0$. Under the alternative hypothesis, Step 2.1 shows that conditional on the planted clique $C$ and $\{V_s, s \in C\}$, $\tP_{V_sV_t}$ is close to the product distribution formed by $\Bern(p)$ for all $s, t \in C$ and $s \neq t$; Step 2.2 shows that conditional on the planted clique $C$, $\tP_{V_s V_s}$ mixed over $\{V_s, s\in C\}$ is close to the product distribution formed by $\Bern(p)$ for all $s \in C$.

%{\bf Step 1:}
First, we show that the null distributions are exactly matched by the reduction scheme.
\prettyref{lmm:TotalVariationBound} implies that $P'_{\ell_s \ell_t}$ and $Q'_{\ell_s \ell_t}$ are well-defined probability measures,
and by definition, $(1-\gamma) Q'_{\ell_s \ell_t}+ \gamma P'_{\ell_s \ell_t} = Q_{\ell_s \ell_t} = \Binom(\ell_s \ell_t,q)$.
Under the null hypothesis, $G \sim \mathcal{G}(n, \gamma)$ and therefore, according to our reduction scheme, $E(V_s, V_t) \sim \Binom(\ell_s \ell_t,q)$ for $s<t$ and $E(V_t, V_t) \sim \Binom(\binom{\ell_t}{2},q)$.
%Recall that $Q_{\ell_s, \ell_t}=\Binom(\ell_s \ell_t,q)$.
 Since the vertices in $V_s$ and $V_t$ are connected uniformly at random such that the total number of edges is $E(V_s,V_t)$, it follows that $\tP_{V_sV_t}=\prod_{(i,j) \in V_s \times V_t} \Bern(q)$ for $s<t$ and $\tP_{V_sV_t} = \prod_{[i,j] \in \calE(V_s)} \Bern(q)$ for $s=t$.  Conditional on $V_1^n$, $\{E(V_s, V_t): 1\le s \le t \le n \}$ are independent and so are $\{\tG_{V_s V_t}:
  1\le s \le t \le n\}$. Consequently, $P_{\tG|H_0^C}=\Prob_0=\prod_{[i,j] \in \calE([N]) } \Bern(q)$ and $\widetilde{G} \sim \mathcal{G}(N,q)$.

%{\bf Step 2:}
Next, we proceed to consider the alternative hypothesis, under which $G$ is drawn from the planted clique model $\mathcal{G}(n,k,\gamma)$. Let $C \subset [n]$ denote the planted clique. Define $S=\cup_{t \in C} V_t$ and recall $K=k \ell$. Then $|S| \sim \Binom(N, K/N)$ and conditional on $|S|$, $S$ is uniformly distributed over all possible subsets of size $|S|$ in $[N]$.
By the symmetry of the vertices of $G$, the distribution of $\tA$ conditional on $C$ does not depend on $C$. Hence, without loss of generality, we shall assume that $C=[k]$ henceforth.
%Let $\tP_{SS}$ denote the distribution of the submatrix $\tA_{\calE(S)}$.
The distribution of $ \widetilde{A}$ can be written as a mixture distribution indexed by the random set $S$ as
\begin{align*}
\widetilde{A} \sim \widetilde{\Prob}_1 \triangleq \mathbb{E}_{S} \left[\tP_{SS} \times   \prod_{ [i,j] \in \calE(S)^c } \Bern(q) \right],
\end{align*}
By the definition of $\Prob_1$,
\begin{align}
%& d_{\rm TV} ( \widetilde{\Prob}_1, \Prob_1 ) \nonumber \\
d_{\rm TV} ( \widetilde{\Prob}_1, \Prob_1 )
&= d_{\rm TV} \left( \mathbb{E}_{S} \left[ \tP_{SS} \times \prod_{[i,j] \in \calE(S)^c } \Bern(q) \right],    \mathbb{E}_{S} \left[  \prod_{[i,j] \in \calE(S) } \Bern(p) \prod_{[i,j] \in \calE(S)^c } \Bern(q)   \right] \right) \nonumber \\
& \le \mathbb{E}_{S} \left[  d_{\rm TV} \left(  \tP_{SS} \times \prod_{[i,j] \in \calE(S)^c } \Bern(q), \prod_{[i,j] \in \calE(S) } \Bern(p) \prod_{[i,j] \in \calE(S)^c } \Bern(q)   \right) \right] \nonumber\\
& =\mathbb{E}_{S} \left[  d_{\rm TV} \left( \tP_{SS}, \prod_{[i,j] \in \calE(S) } \Bern(p)  \right) \right] \nonumber\\
%& \le \mathbb{E}_{C,S} \left[  d_{\rm TV} \left( \mathbb{E}_{\{V_i: i\in C\} } \left [ \prod_{[i,j] \in C \times C } \widetilde{P}_{ij} \biggm \vert S \right], \prod_{[i,j] \in S\times S} \Bern(p)  \right) \1{ 0.5 K \le |S| \le 2K }\right] + \Prob \{ |S|> 2K \} + \Prob\{ |S|<0.5K \} \\
& \le \mathbb{E}_{S} \left[  d_{\rm TV} \left( \tP_{SS}, \prod_{[i,j] \in \calE(S)} \Bern(p)  \right) \1{ |S| \le 1.5K }\right] + \exp (-K/12), \label{eq:TotalvariationDistance}
\end{align}
where the first inequality follows from the convexity of $(P,Q) \mapsto d_{\rm TV} (P,Q)$, and the last inequality follows from applying the Chernoff bound to $|S|$.
Fix an $S \subset [N]$ such that $|S| \le 1.5 K$. Define $P_{V_tV_t}=\prod_{[i,j] \in \calE(V_t)} \Bern(q)$ for $t \in [k]$ and
$P_{V_sV_t}= \prod_{(i,j) \in V_s \times V_t} \Bern(p)$ for $1 \le s<t \le k$.
By the triangle inequality,
\begin{align}
d_{\rm TV} \left( \tP_{SS},
\prod_{[i,j] \in \calE(S)} \Bern(p)   \right)
\le & ~d_{\rm TV} \left( \tP_{SS},  \mathbb{E}_{V_1^k } \left [  \prod_{1 \le s \le t\le k} P_{V_sV_t} \biggm \vert S \right]\right) \label{eq:TotalVarTriangle1} \\
& ~+d_{\rm TV}  \left( \mathbb{E}_{V_1^k} \left [  \prod_{1 \le s \le t\le k} P_{V_sV_t} \biggm \vert S \right], \prod_{[i,j] \in \calE(S)} \Bern(p) \right). \label{eq:TotalVarTriangle2}
\end{align}

%{\bf Step 2.1:}
To bound the term in \prettyref{eq:TotalVarTriangle1}, first note that conditional on $S$, $\{V_1^k\}$ can be generated as follows: Throw balls indexed by $S$ into bins indexed by $[k]$ independently and uniformly at random; let $V_t$ is the set of balls in the $t^\Th$ bin. Define the event $E = \{V_1^k: |V_t| \leq 2 \ell, t \in[k]\}$. Since $|V_t| \sim \Binom(|S|,1/k)$ is stochastically dominated by $\Binom(1.5K,1/k)$ for each fixed $ 1 \le t\le k$, it follows from the Chernoff bound and the union bound that $\Prob\{E^c\} \le k \exp(-\ell/18)$.
\begin{align*}
& d_{\rm TV} \left(  \tP_{SS},  \mathbb{E}_{V_1^k} \left [  \prod_{1 \le s \le t \le k} P_{V_sV_t} \biggm \vert S \right]\right) \nonumber\\
&\overset{(a)}{=} d_{\rm TV} \left(  \mathbb{E}_{V_1^k} \left [  \prod_{1 \le s \le t \le k} \widetilde{P}_{V_s V_t}  \biggm \vert S\right],  \mathbb{E}_{V_1^k} \left [  \prod_{1 \le s \le t \le k} P_{V_sV_t} \biggm \vert S \right]\right) \nonumber \\
& \le \mathbb{E}_{V_1^k } \left [ d_{\rm TV} \left(  \prod_{1 \le s \le t \le k} \widetilde{P}_{V_sV_t},   \prod_{1 \le s \le t \le k} P_{V_sV_t} \right) \biggm \vert S \right] \nonumber\\
%& \overset{(a)}{=} \mathbb{E}_{V_1^k } \left [ d_{\rm TV} \left(  \prod_{1 \le s \le t \le k} \widetilde{P}_{V_s V_t} ,   \prod_{1 \le s \le t \le k} P_{V_sV_t} \right) \biggm \vert S \right] \nonumber \\
& \le \mathbb{E}_{V_1^k } \left [ d_{\rm TV} \left( \prod_{1 \le s \le t \le k} \widetilde{P}_{V_sV_t},   \prod_{1 \le s \le t \le k} P_{V_sV_t} \right)  \1{V_1^k \in E}\biggm \vert S \right] +k \exp(-\ell/18),
\end{align*}
where $(a)$ holds because conditional on $V_1^k$, $\left\{ \tA_{V_s V_t} : s, t \in [k] \right\}$ are independent.  Recall that $\ell_t=|V_t|$. For any fixed $V_1^k \in E$, we have
\begin{align*}
 d_{\rm TV} \left( \prod_{1 \le s \le t \le k} \widetilde{P}_{V_sV_t},   \prod_{1 \le s \le t \le k} P_{V_sV_t} \right)
 &\overset{(a)}{=} d_{\rm TV} \left( \prod_{1\leq s<t \leq k} \widetilde{P}_{V_sV_t},   \prod_{1\leq s<t \leq k} P_{V_sV_t} \right)  \nonumber  \\
 &\overset{(b)}{=} d_{\rm TV} \left( \prod_{1\leq s<t \leq k} P'_{\ell_s \ell_t},  \prod_{1\leq s<t \leq k} P_{\ell_s \ell_t} \right) \nonumber  \\
& \le  d_{\rm TV} \left( \prod_{1\leq s<t \leq k} P'_{\ell_s \ell_t},  \prod_{1\leq s<t \leq k} P_{\ell_s \ell_t} \right) \nonumber \\
& \le \sum_{ 1 \le s< t \le k} d_{\rm TV} \left(   P'_{\ell_s \ell_t},  P_{\ell_s \ell_t} \right) \overset{(c)}{\le} 2 k^2 (8q \ell^2)^{(m_0+1)},
\end{align*}
where $(a)$ follows since $\tP_{V_tV_t}=P_{V_tV_t}$ for all $t \in [k]$;
 $(b)$ is because the number of edges $E(V_s,V_t)$ is a sufficient statistic for testing $\tP_{V_sV_t}$ versus $P_{V_sV_t}$ on the submatrix $A_{V_sV_t}$ of the adjacency matrix; $(c)$ follows from \prettyref{lmm:TotalVariationBound}. Therefore,
\begin{align}
d_{\rm TV} \left(  \tP_{SS},  \mathbb{E}_{V_1^k} \left [  \prod_{1 \le s \le t \le k} P_{V_sV_t} \biggm \vert S \right]\right) \le 2 k^2 (8q \ell^2)^{(m_0+1)} +k \exp(-\ell/18).\label{eq:TotalVariationNondiagonal}
\end{align}

%{\bf Step 2.2:}
To bound the term in \prettyref{eq:TotalVarTriangle2}, applying \prettyref{lmm:chi2tr} yields
\begin{align}
 %& d_{\rm TV}   \left( \mathbb{E}_{V_1^k } \left [ \prod_{[i,j] \in [k]^2 } P_{ij}  \biggm \vert S \right],
%\prod_{[i,j] \in S\times S} \Bern(p) \right) \nonumber\\
& d_{\rm TV}  \left( \mathbb{E}_{V_1^k } \left [ \prod_{1 \le s \le t \le k} P_{V_sV_t}\biggm \vert S \right],\prod_{[i,j] \in \calE(S) } \Bern(p)  \right) \nonumber\\
& \le \frac{1}{2} \prob{E^c} + \frac{1}{2}
\sqrt{\mathbb{E}_{  V_1^k; \tV_1^k} \left[ g (V_1^k , \tV_1^k) \indc{V_1^k \in E}\indc{\tV_1^k \in E} \biggm \vert S  \right] -1 + 2 \prob{E^c}}, \label{eq:mixturebound}
\end{align}
where
\begin{align*}
g (V_1^k , \tV_1^k) &= \int \frac{\prod_{1 \le s \le t \le k} P_{V_sV_t} \prod_{1 \le s \le t \le k} P_{\tV_s \tV_t}} {\prod_{[i,j]\in \calE(S)} \Bern(p) } \\
%&=  \int \frac{\prod_{s=1}^k P_{V_sV_s} \prod_{t=1}^k P_{\tV_t \tV_t}} {\prod_{s,t=1}^k  \Bern(p)^{\otimes \binom{|V_s \cap \tV_t | }{2} }}  \\
& =  \prod_{s,t=1}^k \left( \frac{q^2}{p} + \frac{(1-q)^2}{1-p} \right)^{ \binom{|V_s \cap \tV_t | }{2} } \\
&= \prod_{s,t=1}^k\left( \frac{1-\frac{3}{2} q }{1-2q} \right)^{ \binom{|V_s \cap \tV_t | }{2} }.
\end{align*}
%& \le \chi^2 \left( \mathbb{E}_{V_1, \ldots, V_k } \left [ \prod_{t=1}^k \prod_{[i,j] \in V_t^2} \Bern(q) \times \prod_{\rm else} \Bern(p)  \biggm \vert S \right] \Big\| \prod_{[i,j] \in S\times S} \Bern(p)  \right) \\
%& = \mathbb{E}_{  V_1, \ldots, V_k; \widetilde{V}_1, \ldots, \widetilde{V}_k } \left[ \prod_{i,j=1}^k \left( \frac{1-\frac{3}{2} q }{1-2q} \right)^{ \binom{|V_i \cap \widetilde{V}_j | }{2} } \biggm \vert S  \right] -1,
%\begin{align*}
%&d_{\rm TV}  \left( \mathbb{E}_{V_1^k } \left [ \prod_{t=1}^k \prod_{[i,j] \in V_t^2} \Bern(q) \times \prod_{\rm else} \Bern(p) \biggm \vert S \right],
%\prod_{[i,j] \in S\times S} \Bern(p)  \right) \\
%& = \frac{1}{2} \prob{E^c} + \frac{1}{2}
%\sqrt{\mathbb{E}_{  V_1^k; \tV_1^k} \left[ \prod_{i,j=1}^k \left( \frac{1-\frac{3}{2} q }{1-2q} \right)^{ \binom{|V_i \cap \tV_j | }{2} } \indc{V_1^k \in E}\indc{\tV_1^k \in E} \biggm \vert S  \right] -1 + 2 \prob{E^c}},
%\end{align*}
Let $X \sim \text{Bin}(1.5K, \frac{1}{k^2})$ and $Y \sim \text{Bin}( 3 \ell, e/k)$. It follows that
\begin{align}
& \mathbb{E}_{  V_1^k; \tV_1^k} \left[ \prod_{s,t=1}^k \left( \frac{1-\frac{3}{2} q }{1-2q} \right)^{ \binom{|V_s \cap \tV_t | }{2} } \prod_{s,t=1}^k \indc{|V_s| \leq 2\ell, |\tV_t| \leq 2\ell} \biggm \vert S \right] \nonumber \\
& \overset{(a)}{\leq}   \mathbb{E}_{  V_1^k; \tV_1^k} \left[ \prod_{s,t=1}^k e^{ q \binom{|V_s \cap \tV_t | \wedge 2\ell}{2} } \biggm \vert S \right] \nonumber \\
& \stackrel{(b) }{\leq}  ~ \prod_{s,t=1}^k \expect{ e^{ q \binom{|V_s \cap \tV_t | \wedge 2\ell }{2} }  \biggm \vert S }	\nonumber \\
 & \overset{(c)}{ \le} \left( \mathbb{E} \left[ e^{q\binom{X \wedge 2\ell}{2}} \right ]  \right)^{k^2} \overset{(d)} {\le}    \mathbb{E} \left[ e^{q\binom{Y}{2}} \right]^{k^2} \overset{(e)}{\le}  \exp(72 e^2 q\ell^2), \label{eq:boundchisquare}
\end{align}
where $(a)$ follows from $1+x \le e^x$ for all $x\ge 0$ and $q<1/4$;
$(b)$ follows from the negative association property of $\{|V_s \cap \tV_t |: s,t\in[k]\}$ proved in
 \prettyref{lmm:NA} and \prettyref{eq:NAbound}, in view of the monotonicity of $x \mapsto e^{q\binom{x \wedge 2\ell}{2}}$ on $\reals_+$; $(c)$ follows because $|V_s \cap \tV_t |$ is stochastically dominated by $\Binom(1.5K, 1/k^2)$ for all $(s, t) \in [k]^2$; $(d)$ follows from Lemma~\ref{LemmaComparison};
$(e)$ follows from Lemma~\ref{lmm:decoup} with $\lambda=q/2$ and $q \ell \le 1/8$. Therefore, by \prettyref{eq:mixturebound}
\begin{align}
d_{\rm TV}  \left(  \tP_{SS},\prod_{[i,j] \in \calE(S) } \Bern(p)  \right) & \le0.5 k e^{-\frac{\ell}{18}} + 0.5 \sqrt{ e^{72e^2 q \ell^2} -1 + 2  k e^{-\frac{\ell}{18}}} \nonumber\\
& \le 0.5 k e^{-\frac{\ell}{18}} + 0.5 \sqrt{e^{72e^2 q \ell^2} -1} + \sqrt{0.5k} e^{-\frac{\ell}{36}}. \label{eq:TotalVariationDiagonal}
\end{align}
The proposition follows by combining \prettyref{eq:TotalvariationDistance}, \prettyref{eq:TotalVarTriangle1}, \prettyref{eq:TotalVarTriangle2}, \prettyref{eq:TotalVariationNondiagonal} and \prettyref{eq:TotalVariationDiagonal}.
\end{proof}

\subsection{Proof of \prettyref{prop:test}}
	\label{sec:pf-test}
	  \begin{proof}
  By assumption the test $\phi$ satisfies
  \begin{align*}
  \Prob_0 \{  \phi(G') =1 \} + \Prob_1 \{ \phi(G')=0 \} = \eta,
  \end{align*}
  where $G'$ is the graph in $\PDS(N,K,2q,q)$ distributed according to either $\Prob_0$ or $\Prob_1$.
    Let $G$ denote the graph in the $\PC(n,k,\gamma)$ and $\tG$ denote the corresponding output of the randomized reduction scheme.
  %  Recall that $h$ is the random mapping from $A$ to $\tA$.
  \prettyref{prop:reduction} implies that $\tG \sim \mathcal{G}(N,q)$ under $H_0^{\rm C}$. Therefore $\Prob_{H_0^{\rm C}} \{  \phi( \tG ) =1\} = \Prob_0 \{  \phi(G') =1 \}$.
  Moreover,
  \begin{align*}
  | \Prob_{H_1^{\rm C}} \{ \phi( \tG ) =0 \} - \Prob_{1} \{ \phi(G')=0  \} | \le d_{\rm TV} (P_{\tG|H_1^C}, \Prob_1) \le \xi.
  \end{align*}
%    where $\widetilde{\Prob}_1$ is the distribution of $\tA$ under $\Prob_{H_1^{\rm C}}$.
It follows that
  \begin{align*}
  \Prob_{H_0^{\rm C}} \{  \phi( \tG ) =1\} + \Prob_{H_1^{\rm C} } \{\phi(\tG ) =0 \} &
%  \le  \Prob_0 \{ \phi (A')=1 \} + \Prob'_1 \{ \phi(A')=0 \}   + \xi
  \le \eta+ \xi.
  \end{align*}
%  \nb{what the hell is $A'$...}
  \end{proof}

\subsection{Proof of \prettyref{thm:main}}
\label{sec:pf-main}

\begin{proof}
Fix $ \alpha>0$ and $0<\beta<1$ that satisfy \prettyref{eq:HardRegimeExpression}. Then it is straightforward to verify that
    \begin{align}
    \alpha<\beta< \min \left \{  \frac{2+m_0\delta }{4+2\delta} \alpha, \; \frac{1}{2}-\delta + \frac{1+2\delta}{4+2\delta} \alpha \right\}
    \label{eq:betaub}
    \end{align}
holds for some $\delta>0$.
%The supremum in \prettyref{eq:HardRegimeExpression} is attained when $\frac{2+m_0\delta }{4+2\delta} \alpha=\frac{1}{2}-\delta + \frac{1+2\delta}{4+2\delta} \alpha$, i.e.,
%\begin{align*}
%\delta= \frac{\sqrt{[ (m_0-2)\alpha+3]^2 + 8(2-\alpha) }- (m_0-2)\alpha-3 }{4}.
%\end{align*}
%Then it follows from the assumption \prettyref{eq:HardRegimeExpression} that
%\begin{align}
%\beta<  \frac{2+m_0\delta }{4+2\delta} \alpha =\frac{1}{2}-\delta + \frac{1+2\delta}{4+2\delta} \alpha. \label{eq:betaub}
%\end{align}
Let $\ell \in \naturals$ and $q_\ell=\ell^{-(2+\delta)}$.
Define
    \begin{align}
     n_\ell= \lfloor \ell^{\frac{2+\delta}{\alpha}-1 } \rfloor, \; k_\ell=\lfloor \ell^{\frac{(2+\delta)\beta}{\alpha}-1 } \rfloor,\; N_\ell=n_\ell\ell, \; K_\ell=k_\ell\ell. \label{eq:DefNK}
    \end{align}
Then
\begin{align}
\lim_{\ell \to \infty} \frac{\log \frac{1}{q_\ell} }{ \log N_\ell} = \frac{(2+\delta)}{(2+\delta)/\alpha-1 +1 } =\alpha, \;
\lim_{\ell \to \infty}  \frac{\log K_\ell}{ \log N_\ell}  = \frac{(2+\delta)\beta/\alpha-1 +1 }{(2+\delta)/\alpha-1 +1} = \beta. \label{eq:NKq}
\end{align}
Suppose that for the sake of contradiction there exists a small $\epsilon>0$ and a sequence of randomized polynomial-time tests $\{\phi_\ell\}$ for $\PDS(N_\ell,K_\ell,2q_\ell,q_\ell)$, such that
\begin{align*}
\Prob_0 \{ \phi_{N_\ell,K_\ell}(G') =1 \} + \Prob_1 \{ \phi_{N_\ell,K_\ell}(G')=0 \} \le 1-\epsilon
\end{align*}
holds for arbitrarily large $\ell$, where $G'$ is the graph in the $\PDS(N_\ell,K_\ell,2q_\ell,q_\ell)$.
Since $\beta>\alpha$, we have $k_{\ell} \ge \ell^{1+\delta}$.
Therefore, $16 q_{\ell} \ell^2 \le 1$ and $k_{\ell} \ge 6e\ell$ for all sufficiently large $\ell$. Applying \prettyref{prop:test}, we conclude that $G \mapsto \phi(\tG)$ is a randomized polynomial-time test for $\PC(n_\ell,k_\ell, \gamma)$ whose Type-I+II error probability satisfies
\begin{align}
\Prob_{H_0^{\rm C}} \{  \phi_{\ell}( \tG ) =1\} + \Prob_{H_1^{\rm C} } \{\phi_\ell(\tG ) =0 \} \le 1- \epsilon+ \xi,
% \le \frac{1}{2}+o(1) < \frac{2}{3}
\label{eq:PlantedCliqueContradiction}
\end{align}
where $\xi$ is given by the right-hand side of \prettyref{eq:defxi}. By the definition of $q_\ell$, we have $q_\ell \ell^2 =\ell^{-\delta}$ and thus
\begin{align*}
k_\ell^2 (q_\ell\ell^2)^{m_0+1} \le \ell^{2 \left( (2+\delta)\beta/\alpha -1 \right) - (m_0+1) \delta}  \le \ell^{-\delta},
\end{align*}
where the last inequality follows from \prettyref{eq:betaub}. Therefore $\xi\to0$ as $\ell\diverge$. Moreover, by the definition in \prettyref{eq:DefNK},
\begin{align*}
% \frac{\beta}{2} \le
  \lim_{\ell \to \infty} \frac{\log k_\ell}{ \log n_\ell} = \frac{(2+\delta)\beta/\alpha-1 }{(2+\delta)/\alpha-1}\le 1- \delta,
\end{align*}
where the above inequality follows from \prettyref{eq:betaub}. Therefore, \prettyref{eq:PlantedCliqueContradiction} contradicts our assumption that \prettyref{hyp:HypothesisPlantedClique} holds for $\gamma$.
Finally, if \prettyref{hyp:HypothesisPlantedClique} holds for any $\gamma>0$, \prettyref{eq:computationallimit} follows from \prettyref{eq:HardRegimeExpression} by sending $\gamma \downarrow 0$.
\end{proof}

\section{Computational Lower Bounds for Approximately Recovering a Planted Dense Subgraph with Deterministic Size} \label{sec:PDSRecovery}
Let $\widetilde{\calG}(N, K, p, q)$ denote the planted dense subgraph model with $N$ vertices and a deterministic dense subgraph size $K$:
(1) A random set $S$ of size $K$ is uniformly chosen from $[N]$; (2) for any two vertices, they are connected with probability $p$
if both of them are in $S$ and with probability $q$ otherwise, where $p>q$. Let \PDSR($n,K,p,q,\epsilon$) denote the planted dense subgraph recovery problem,
where given a graph generated from $\widetilde{\calG}(N, K, p, q)$ and an $\epsilon<1$, the task is to output a set $\widehat{S}$ of size $K$
such that $\widehat{S}$ is a $(1-\epsilon)$-approximation of $S$, \ie, $| \widehat{S} \cap S | \ge (1-\epsilon) K$.
The following theorem implies that \PDSR($N,K,p=cq,q,\epsilon$) is at least as hard as \PDS$(N,K,p=cq,q)$ if $K q =\Omega(\log N)$.
Notice that in \PDSR($N,K,p,q,\epsilon$), the  planted dense subgraph has a deterministic size $K$, while in \PDS$(N,K,p,q)$, the size of the planted dense subgraph is binomially distributed with mean $K$.
\begin{theorem} \label{thm:recoversinglecluster}
For any constant $\epsilon<1$ and $c>0$, suppose there is an algorithm $\calA_N$ with running time $T_N$ that solves
the \PDSR($N,K,cq,q,\epsilon$) problem with probability $1-\eta_N$.
Then there exists a test $\phi_N$ with running time at most $N^2+NT_N+NK^2$ that solves the \PDS$(N,2K,cq,q)$ problem with Type-I+II error probabilities at most $\eta_N+ e^{-C K} + 2N e^{-C K^2 q + K \log N}$, where the constant $C>0$ only depends on $\epsilon$ and $c$.
\end{theorem}
\begin{proof}
%Given $G$ generated either under $\mathcal{G}(N,q)$ or $\mathcal{G}(N,2K,p=cq, q)$, we obtain a sequence of $N$ graphs $G_1, \ldots, G_N$
%by each time picking  a vertex  (without replacement) in any arbitrary order and replacing it with a new vertex that connects to
%all other vertices independently at random with probability $q$.
Given a graph $G$, we construct a sequence of graphs $G_1, \ldots, G_N$ sequentially as follows: Choose a permutation $\pi$ on the $N$ vertices uniformly at random. Let $G_0=G$. For each $t \in [N]$, replace the vertex $\pi(t)$ in $G_{t-1}$ with a new vertex that connects to all other vertices independently at random with probability $q$.
We run the given algorithm $\calA_N$ on $G_1, \ldots, G_N$ and let $S_1, \ldots, S_N$
denote the outputs which are sets of $K$ vertices. Let $E(S_i,S_i)$ denote the total number of edges in $S_i$ and $\tau= q+ (1-\epsilon)^2(p-q)/2$.
 Define a test $\phi: G \to \{0,1\}$ such that $\phi(G)=1$ if and only if $\max_{i \in [N]} E(S_i,S_i)> \tau \binom{K}{2}$.
% ; otherwise $\phi(G)=0$.
 The construction of each $G_i$ takes $N$ time units; the running time of $\calA$ on $G_i$ is at most $T_N$ time units; the computation of $E(S_i,S_i)$ takes at most $K^2$ time units.
 Therefore, the total running time of $\phi$ is at most $N^2+N T_N+NK^2$.

 Next we upper bound the Type-I and II error probabilities of $\phi$. Let $C=C(\epsilon,c)$ denote a positive constant
 whose value may depend on the context.
% may change line by line.
If $G \sim  \mathcal{G}(N,q)$, then all $G_i$ are distributed according to $\mathcal{G}(N,q)$. By the union bound and the Bernstein inequality,
\begin{align*}
\Prob_0 \{ \phi(G)=1\} &\le \sum_{i=1}^N \Prob_0 \left\{ E(S_i,S_i) \ge \tau \binom{K}{2} \right\} \\
& \le  \sum_{i=1}^N \sum_{S': S' \subset [N], |S'| =K} \Prob_0 \left \{ E(S', S' ) \ge \tau \binom{K}{2} \right\} \\
 & \le N \binom{N}{K} \exp \left(  - \frac{\binom{K}{2}^2 (1-\epsilon)^4 (p-q)^2 /4}{ 2 \binom{K}{2}q + \binom{K}{2} (1-\epsilon)^2 (p-q)/3 } \right) \\
 & \le N \exp (  -C K^2 q + K \log N).
\end{align*}
If $G \sim \mathcal{G}(N,2K,p,q)$, let $S$ denote the set of vertices in the planted dense subgraph.
Then $|S| \sim \Binom(N, \frac{2K}{N})$ and by the Chernoff bound, $\Prob_1[|S| < K] \le \exp(-C K)$. If $|S| =K' \ge K$, then there must exist some $I \in [N]$ such that $G_I$ is distributed exactly as $ \widetilde{ \calG}(N,K,p, q)$. Let $S^\ast$ denote the set of vertices in the planted dense subgraph of $G_I$ such that $|S^\ast |= K$. Then
conditional on $I=i$ and the success of $\calA_N$ on $G_i$, $|S_i \cap S^\ast| \ge (1-\epsilon)K$.
Thus by the union bound and the Bernstein inequality, for $K' \ge K$,
\begin{align*}
%&\Prob_1\left\{ E(S_i, S_i) < \tau \binom{K}{2} \bigg| |S|=K', I =i \right \} \\
&\Prob_1\{  \phi(G)=0 | |S| = K', I=i \} \\
%& \le \Prob_1\left\{ E(S_i \cap S^\ast, S_i \cap S^\ast) < \tau \binom{K}{2} \bigg| |S|=K', I =i \right \} \\
& \le \eta_N+ \sum_{S' \subset [N]: |S'| =K, |S' \cap S^\ast| \ge (1-\epsilon)K } \Prob_1\left\{ E(S', S') \le \tau \binom{K}{2} \bigg| |S|=K', I =i \right \} \\
& \le \eta_N + \sum_{t \ge (1-\epsilon)K}^{K} \binom{K}{t} \binom{N-K}{K-t} \exp \left(  - \frac{\binom{K}{2}^2 (1-\epsilon)^4 (p-q)^2 /4}{ 2 \binom{K}{2}p + \binom{K}{2} (1-\epsilon)^2 (p-q)/3 } \right) \\
& \le \eta_N + K \exp (  -C K^2 q+ K \log N).
\end{align*}
It follows that
\begin{align*}
&\Prob_1 \{ \phi(G)=0\} \\
& \le \Prob_1\{|S| < K\} + \sum_{K' \ge K} \sum_{i=1}^N \Prob_1\{ |S| =K', I=i \} \Prob_1\{  \phi(G)=0 | |S| = K', I=i \}  \\
& \le  \exp(-C  K) +   \eta_N + K \exp (  -C K^2 q+ K \log N).
\end{align*}
\end{proof}

\section{A Lemma on Hypergeometric Distributions}

\label{app:H}

\begin{lemma}
There exists a function $\tau : \reals_+ \to \reals_+$ satisfying $\tau(0+) = 1$ such that the following holds:
For any $p \in \naturals$ and $m \in [p]$, let $H \sim \Hyper(p,m,m)$ and $\lambda = b \pth{\frac{1}{m} \log \frac{\eexp p}{m} \wedge \frac{p^2}{m^4}}$ with $ 0<b< 1/(16 \eexp)$. Then
\begin{equation}
\expect{\exp\pth{\lambda H^2}} \leq \tau(b).
	\label{eq:H}
\end{equation}
%for some absolutely constant $B$. % only depends on $a$.
	\label{lmm:H}
\end{lemma}

	\begin{proof}
Notice that if $p \le 64$, then the lemma trivially holds. Hence, assume $p \ge 64$ in the rest of the proof.
We consider three separate cases depending on the value of $m$. We first deal with the case of $m \geq \frac{p}{4}$. Then $\lambda = \frac{b p^2}{m^4} \leq \frac{256b}{p^2}$. Since $H \leq p$ with probability $1$, we have $\expect{\exp\pth{\lambda H^2}} \leq \exp(256b)$.

Next assume that $m \leq \log \frac{\eexp p}{m}$. Then $ m \le \log p$ and $\lambda=  \frac{b}{m} \log \frac{\eexp p}{m} $.
Let $(\ntok{s_1}{s_m}) \, \iiddistr \, \Bern(\frac{m}{p-m})$. Then $S = \sum_{i=1}^m  s_i \sim \text{Bin}(m,\frac{m}{p-m})$ which dominates $H$ stochastically.  It follows that
%	Following the same steps in \prettyref{eq:hbin} -- \prettyref{eq:case2}, we have
	\begin{align}
	\expect{\exp\pth{\lambda H^2}}  & \leq \expect{\exp\pth{\lambda m S}} \nonumber \\
& = \left[ 1+ \frac{m}{p-m}\left( \eexp^{\lambda m} -1 \right) \right]^m  \nonumber \\
 & \overset{(a)} {\leq} \exp \left( \frac{2m^2}{p} \left( \left( \frac{\eexp p}{m} \right)^b -1  \right) \right) \nonumber \\
 & \overset{(b)}{\leq} \exp \left( \frac{2 (\log p)^2 }{p} \left( \left( \frac{\eexp p}{\log p} \right)^b -1  \right)   \right) \nonumber \\
 & \overset{(c)} {\leq} \max_{1 \le p \le 512} \left \{ \exp \left( \frac{2 (\log p)^2 }{p} \left( \left( \frac{\eexp p}{\log p} \right)^b -1  \right)   \right) \right\} : =\tau ( b),
	\label{eq:mlogp}
\end{align}
where $(a)$ follows because $1+x \le \exp(x)$ for all $x \in \reals$ and $m \le p/2$; $(b)$ follows because $m \le \log p$
and $f(x)=\frac{2x^2}{p} \left( \left(\frac{\eexp p}{x} \right)^b-1 \right)$ in non-decreasing in $x$; $(c)$ follows because $g(x)= \frac{2 (\log x)^2 }{x} \left[ \left( \frac{\eexp x}{\log x} \right)^b -1  \right]$ is non-increasing when $x \ge 512 $; $\tau(0+)=1$ by definition.
	
In the rest of the proof we shall focus on the intermediate regime: $\log \frac{\eexp p}{m} \leq m \leq \frac{p}{4}$.	%Choose
%	\begin{equation}
%	b \leq \frac{1}{32 (\eexp-1)^2} \wedge \frac{1}{160 \eexp}.
%	\label{eq:b-value}
%\end{equation}
%	$\Hyper(p,m,m)$ is stochastically dominated by the binomial distribution $\text{Bin}(m,\frac{m}{p-m})$,
Since $S $ dominates $H$ stochastically,
\begin{equation}
	\expect{\exp\pth{\lambda H^2}} \leq \expect{\exp\pth{\lambda S^2}}.
	\label{eq:MS}
\end{equation}
Let $(\ntok{t_1}{t_m}) \, \iiddistr \, \Bern(\frac{m}{p-m})$ and $T = \sum_{i=1}^m  t_i$, which is an independent copy of $S$. Next we use a decoupling argument to replace $S^2$ by $ST$:
\begin{align}
\pth{\expect{\exp\pth{\lambda S^2}}}^2
= & ~ \left( \Expect\bqth{\exp\bpth{\lambda \sum_{i=1}^m s_i^2 + \lambda \sum_{i \neq j} s_i s_j}}	\right)^2 \nonumber \\
\leq & ~ \expect{\exp\pth{2 \lambda S}} \Expect\bqth{\exp\bpth{2 \lambda \sum_{i \neq j} s_i s_j}}  \label{eq:CS} \\
\leq & ~ \expect{\exp\pth{2 \lambda S}} \Expect\bqth{\exp\bpth{8 \lambda \sum_{i \neq j} s_i t_j}}, \label{eq:decouple2}\\
\leq & ~ \expect{\exp\pth{2 \lambda S}} \expect{\exp\pth{8 \lambda S T}}, \label{eq:Ssq}
\end{align}
	where \prettyref{eq:CS} is by Cauchy-Schwartz inequality and \prettyref{eq:decouple2} is a standard decoupling inequality (see, \eg, \cite[Theorem 1]{Vershynin11}).

	The first expectation on the right-hand side \prettyref{eq:Ssq}	can be easily upper bounded as follows: Since $m \geq \log \frac{\eexp p}{m}$, we have $\lambda \leq b$.
%	Similar to \prettyref{eq:mlogp} and
	Using the convexity of the exponential function:
	\begin{equation}
	\exp(ax)-1 \leq (\eexp^a-1)x, \quad x\in[0,1],
	\label{eq:expcvx}
    \end{equation}	
we have
	\begin{align}
	\expect{\exp\pth{2 \lambda S}}
\leq & ~ \exp\pth{\frac{m^2}{p-m} \pth{\eexp^{2\lambda}-1}} \leq \exp\pth{\frac{ 4 \left(  \eexp^{2b} -1 \right) m^2 \lambda }{bp} }	\nonumber \\
\leq & ~ 	\exp\pth{4 \left( \eexp^{2b} -1 \right) \frac{ m \log \frac{\eexp p}{m} }{p} } \leq \exp\pth{ 4 \left( \eexp^{2b} -1 \right)}, 	 \label{eq:expS}
\end{align}
where the last inequality follows from $\max_{0 \leq x \leq 1}x \log \frac{\eexp}{x} = 1$.
	
	Next we prove that for some function $\tau' : \reals_+ \to \reals_+$ satisfying $\tau'(0+) = 1$,
	\begin{equation}
	\expect{\exp\pth{8 \lambda S T}} \le \tau'(b),
	\label{eq:expST}
\end{equation}
which, in view of \prettyref{eq:MS}, \prettyref{eq:Ssq} and \prettyref{eq:expS}, completes the proof of the lemma. We proceed toward this end by truncating on the value of $T$. First note that
\begin{equation}
\expect{\exp\pth{8 \lambda S T} \indc{T > \frac{1}{8\lambda}}} \leq \expect{\exp\pth{8 b  T \log \frac{\eexp p}{m}}  \indc{T > \frac{1}{8\lambda}}}
	\label{eq:expST.trunc2}
\end{equation}
where the last inequality follows from $S \le m$ and $\lambda m \le b \log \frac{\eexp p}{m}$. It follows from the definition that
\begin{align}
&\expect{\exp\pth{8 b   T \log \frac{\eexp p}{m}} \indc{T > \frac{1}{8\lambda}}}  \nonumber \\
& \le \sum_{ t \ge 1/ (8 \lambda) } \exp \left( 8b t \log \frac{\eexp p}{m} \right) \binom{m}{t} \left( \frac{m}{p-m} \right)^t \nonumber \\
& \overset{(a)}{\le} \sum_{ t \ge 1/ (8 \lambda) } \exp \left( 8b t\log \frac{\eexp p}{m}  + t \log \frac{\eexp m}{t} -t \log \frac{p}{2m}  \right) \nonumber\\
& \overset{(b)}{\le} \sum_{ t \ge 1/ (8 \lambda) } \exp \left( 8bt \log \frac{\eexp p}{m}  + t \log \left( 8 \eexp b \log \frac{\eexp p}{m} \right) -t \log \frac{p}{2m}  \right) \nonumber \\
& \overset{(c)}{\le} \sum_{ t \ge 1/ (8 \lambda) } \exp \left[ - t \left( \log 2 - 8 b \log (4\eexp)  - \log \left( 8 \eexp b \log (4 \eexp) \right) \right) \right] \nonumber \\
& \overset{(d)}{\le} \sum_{ t \ge 1/ (8 b) } \exp \left[ - t \left( \log 2 - 8 b \log (4\eexp)  - \log \left( 8 \eexp b \log (4 \eexp) \right) \right)
\right] := \tau''(b)
\label{eq:expST-cond}
\end{align}
where $(a)$ follows because $\binom{m}{t} \le \left( \frac{em}{t} \right)^t$ and $m \le p/2$; $(b)$ follows because $ \frac{m}{t} \le 8  m\lambda \le 8 b \log \frac{\eexp p}{m} $; $(c)$ follows because $ m \le p/4$ and $ b \le 1/(16 \eexp)$; $(d)$ follows because $\lambda \le b$; $\tau''(0+)=0$ holds because $\log 2< 8 b \log (4\eexp)  + \log \left( 8 \eexp b \log (4 \eexp) \right)$ for $b \le 1/(16\eexp)$.

Recall that $m \geq \log \frac{\eexp p}{m}$. Then $\lambda = b \pth{\frac{1}{m} \log \frac{\eexp p}{m}  \wedge \frac{p^2}{m^4}} \leq b \pth{1 \wedge \frac{p^2}{m^4}}$. Hence, we have
\begin{equation}
\frac{m^2 \lambda}{p}  \leq b \pth{\frac{m^2}{p} \wedge \frac{p}{m^2}} \leq b
	\label{eq:b-1}
\end{equation}
By
%using the stochastic dominance relationsh \prettyref{eq:hyperbin},
conditioning on $T$ and averaging with respect to $S$, we have
%\begin{equation}
%\expect{\exp\pth{8 \lambda S T} \indc{T \leq \frac{1}{8\lambda}}} \leq \expect{\exp\pth{\frac{2m^2}{p} \pth{\exp(8\lambda T)-1}} \indc{T \leq \frac{1}{8\lambda}}}	.
%	\label{eq:expST-cond}
%\end{equation}
%Define $L = \frac{1}{8\lambda}$. Then
%Next we bound the right-hand side of \prettyref{eq:expST-cond} by truncating the range of $T$:
%Recall that $m \geq \log \frac{\eexp p}{m}$. Then $\lambda = b \pth{\frac{1}{m} \log \frac{\eexp p}{m}  \wedge \frac{p^2}{m^4}} \leq b \pth{1 \wedge \frac{p^2}{m^4}}$. In view of our choice of $b$ in \prettyref{eq:b-value}, we have
%\begin{equation}
%\frac{m^2 \lambda}{p}  \leq b \pth{\frac{m^2}{p} \wedge \frac{p}{m^2}} \leq b \leq \frac{1}{32(\eexp-1)^2}.
%	\label{eq:b-1}
%\end{equation}
%Repeatedly using \prettyref{eq:expcvx}, we have
\begin{align}
\expect{\exp\pth{8 \lambda S T} \indc{T \leq \frac{1}{8\lambda}}}  \leq & ~\expect{\exp\pth{\frac{2m^2}{p} \pth{\exp(8\lambda T)-1}} \indc{T \leq \frac{1}{8\lambda}}} \nonumber \\
\overset{(a)}{\leq} & ~ \expect{\exp\pth{\frac{16\eexp m^2}{p} \lambda T}}  \nonumber \\
\overset{(b)}{\leq} & ~ \exp \sth{\frac{2 m^2}{p} \pth{\exp\pth{\frac{16\eexp m^2 \lambda}{p} } - 1} } \nonumber \\
\overset{(c)}{\leq} & ~ \exp \sth{\frac{32 \eexp^2 m^4}{p^2} \lambda}  \overset{(d)} {\leq}  ~ \exp ( 32 \eexp^2 b)  	\label{eq:expST.trunc1},
\end{align}
where $(a)$ follows from $\eexp^x -1 \le e^a x$ for $ x\in [0,a]$; $(b)$ follows because $T \sim \text{Bin}(m,\frac{m}{p-m})$ and $p \geq 2m$; $(c)$ follows due to~\prettyref{eq:b-1} and $16 \eexp b \le 1$; $(d)$ follows because $\lambda \le b \frac{p^2}{m^4}$.
% \prettyref{eq:a2} is due to , and \prettyref{eq:a3} -- \prettyref{eq:expST.trunc1} follow from  \prettyref{eq:b-1}.
 Assembling \prettyref{eq:expST.trunc2}, \prettyref{eq:expST-cond} and \prettyref{eq:expST.trunc1}, we complete the proof of \prettyref{eq:expST}, hence the lemma.
\end{proof}

\end{document}